\documentclass{article}
\usepackage[english,francais]{babel}
\usepackage[T1]{fontenc}
\usepackage[utf8]{inputenc}
\usepackage{amsmath,amssymb,mathrsfs}
\usepackage{mathrsfs}
\usepackage{multicol}
\usepackage[colorlinks=true,citecolor=red,linkcolor=red,urlcolor=blue]{hyperref}
\usepackage{pgf}
\usepackage{amsthm}
\usepackage{pinlabel}

\newtheorem{theo}{Théorème}[section]
\newtheorem*{theo*}{Théorème}
\newtheorem*{prop*}{Proposition}
\newtheorem{prop}[theo]{Proposition}
\newtheorem{lemme}[theo]{Lemme}
\newtheorem*{definition}{Définition}

\newtheorem{corollaire}[theo]{Corollaire}

\newtheorem{sous-lemme}[theo]{Sous-lemme}

\newcommand{\Sph}{\mathbb{S}}

\newcommand{\G}{\mathcal{G}}
\newcommand{\E}{\mathcal{E}}
\newcommand{\C}{\mathcal{C}}

\newcommand{\R}{\mathbb{R}}
\newcommand{\N}{\mathbb{N}}
\newcommand{\Z}{\mathbb{Z}}

\newlength{\larg}
\begin{document}
\title{Hyperbolicité du graphe des rayons et quasi-morphismes sur un gros groupe modulaire}
\author{Juliette Bavard\footnote{Soutenue par l'attribution d'une allocation doctorale Région \^Ile-de-France.}}
\maketitle

\selectlanguage{english}
\begin{abstract} The mapping class group $\Gamma$ of the complement of a Cantor set in the plane arises naturally in dynamics. We show that the ray graph, which is the analog of the complex of curves for this surface of infinite type, has infinite diameter and is hyperbolic. We use the action of $\Gamma$ on this graph to find an explicit non trivial quasimorphism on $\Gamma$ and to show that this group has infinite dimensional second bounded cohomology. Finally we give an example of a hyperbolic element of $\Gamma$ with vanishing stable commutator length. This carries out a program proposed by Danny Calegari.
\end{abstract}

\selectlanguage{francais}
\begin{abstract} Le groupe modulaire $\Gamma$ du plan privé d'un ensemble de Cantor apparaît naturellement en dynamique. On montre ici que le graphe des rayons, analogue du complexe des courbes pour cette surface de type infini, est de diamètre infini et hyperbolique. On utilise l'action de $\Gamma$ sur ce graphe hyperbolique pour exhiber un quasi-morphisme non trivial explicite sur $\Gamma$ et pour montrer que le deuxième groupe de cohomologie bornée de $\Gamma$ est de dimension infinie. On donne enfin un exemple d'un élément hyperbolique de $\Gamma$ dont la longueur stable des commutateurs est nulle. Ceci réalise un programme proposé par Danny Calegari. \end{abstract}

\newpage

\tableofcontents

\newpage

\section{Introduction}
\subsection{\og Gros \fg \ groupes modulaires et dynamique}

Lorsque $S$ est une surface orientable de type fini ou infini, c'est-à-dire une variété compacte orientable de dimension $2$ éventuellement privée d'un nombre fini ou infini de points, le \emph{groupe modulaire} de $S$, noté $MCG(S)$ pour \og Mapping Class Group\fg, est le groupe des classes d'isotopies d'homéomorphismes de $S$ préservant l'orientation.
Si l'on connait aujourd'hui de nombreuses caractéristiques des groupes modulaires des surfaces compactes privées d'un nombre fini de points, les groupes modulaires des surfaces de type infini sont beaucoup moins étudiés. Pourtant, comme le souligne Danny Calegari sur son blog \og Big mapping class groups and dynamics \fg (voir \cite{blog-Calegari}), ces \og gros \fg \  groupes modulaires apparaissent naturellement en dynamique, en particulier à travers la construction suivante.

On note $Homeo^+(\R^2)$ le groupe des homéomorphismes du plan préservant l'orientation. Soit $G$ un sous-groupe de $Homeo^+(\R^2)$. Si l'orbite $G.p$ d'un point $p \in \R^2$ est bornée, alors il existe un morphisme de $G$ vers $MCG(\R^2-K)$, où $K$ est soit un ensemble fini, soit un ensemble de Cantor.

En effet, la réunion $\tilde K$ de l'adhérence de l'orbite $G.p$ avec l'ensemble des composantes connexes bornées de son complémentaire est un ensemble compact, invariant par $G$ et de complémentaire connexe. Le groupe $G$ agit sur le quotient du plan obtenu en \og écrasant \fg \ chacune des composantes connexes de $\tilde K$ sur un point (un point par composante), qui est encore homéomorphe au plan. L'image de $\tilde K$ au quotient est un sous-ensemble $K$ du plan, totalement discontinu. Quitte à remplacer $K$ par l'un de ses sous-ensembles bien choisi, on peut supposer que $K$ est un ensemble minimal, c'est-à-dire tel que toute orbite $G.q$ avec $q\in K$ est dense dans $K$. Comme $K$ est compact, on sait alors que c'est soit un ensemble fini, soit un ensemble de Cantor. On obtient par cette construction un morphisme de $G$ vers $MCG(\R^2-K)$.

Le groupe modulaire de $\R^2$ privé d'un nombre fini de points, qui a pour sous-groupe d'indice fini le quotient d'un groupe de tresses par son centre, a été très étudié. C'est le cas où $K$ est un ensemble de Cantor qui va nous intéresser ici. On notera : $$\Gamma:=MCG(\R^2-Cantor).$$

Dans \cite{Calegari-Circular}, Calegari montre qu'il existe un morphisme injectif de $\Gamma$ dans $Homeo^+(\Sph^1)$. C'est en particulier la première étape pour montrer qu'un sous-groupe de difféomorphismes du plan préservant l'orientation et ayant une orbite bornée est circulairement ordonnable. Dans le but d'établir de nouvelles propriétés sur le groupe $\Gamma$, on réalise ici un programme proposé par Calegari dans \cite{blog-Calegari}.

\subsection{Graphe des rayons}
Un objet central dans l'étude des groupes modulaires des surfaces de type fini est \textit{le complexe des courbes}, un complexe simplicial associé à chaque surface, dont les simplexes sont les ensembles de classes d'isotopies de courbes simples essentielles sur la surface qui peuvent être réalisées par des représentants disjoints. L'hyperbolicité de ce complexe, établie par Howard Masur et Yair Minsky (voir \cite{Masur-Minsky}), a permis de grandes avancées dans l'étude de ces groupes. Dans le cas du groupe $\Gamma$ que l'on considère, le complexe des courbes de $\R^2$ privé d'un ensemble de Cantor n'est pas intéressant du point de vue de la géométrie à grande échelle introduite par Gromov, car il est de diamètre $2$. Danny Calegari propose de remplacer ce complexe par le \emph{graphe des rayons}, qu'il définit de la manière suivante (voir figure \ref{figu:graphe-rayons} pour des exemples de rayons) :

\begin{definition} [Calegari]
Le \emph{graphe des rayons} est le graphe dont les sommets sont les classes d'isotopies des arcs simples joignant l'infini à un point de l'ensemble de Cantor $K$ et d'intérieur inclus dans le complémentaire de $K$, appelés \emph{rayons}, et dont les arêtes sont les paires de tels rayons qui ont des représentants disjoints.
\end{definition}

\pgfdeclareimage[interpolate=true,height=5cm]{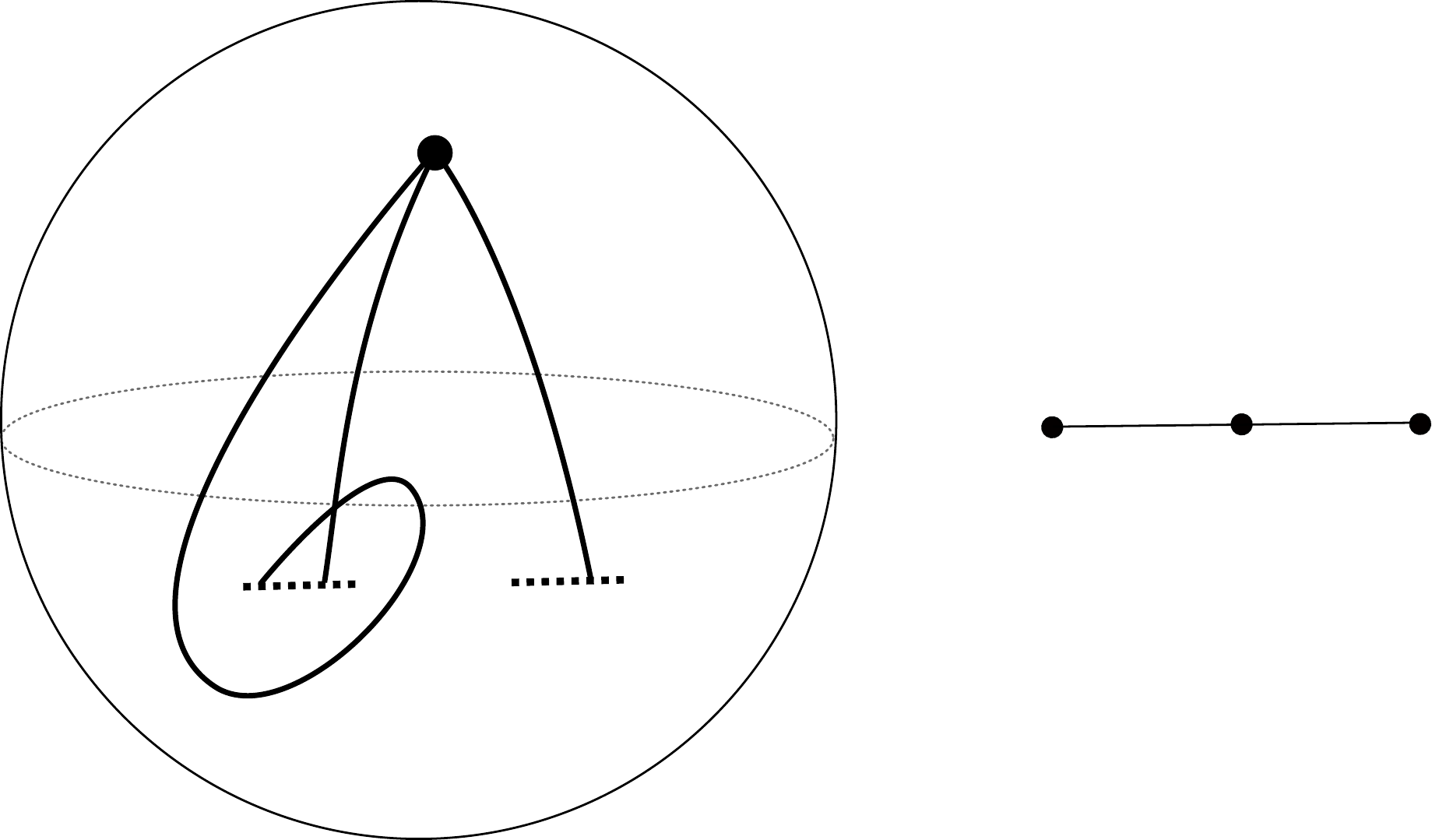}{graphe-rayons}
\begin{figure}[!h]
\labellist
\pinlabel $\alpha$ at 95 190
\pinlabel $\beta$ at 148 190
\pinlabel $\gamma$ at 202 190
\pinlabel $\alpha$ at 392 137
\pinlabel $\gamma$ at 458 136
\pinlabel $\beta$ at 526 134
\pinlabel $\infty$ at 162 270
\endlabellist
\centering
\includegraphics[scale=0.5]{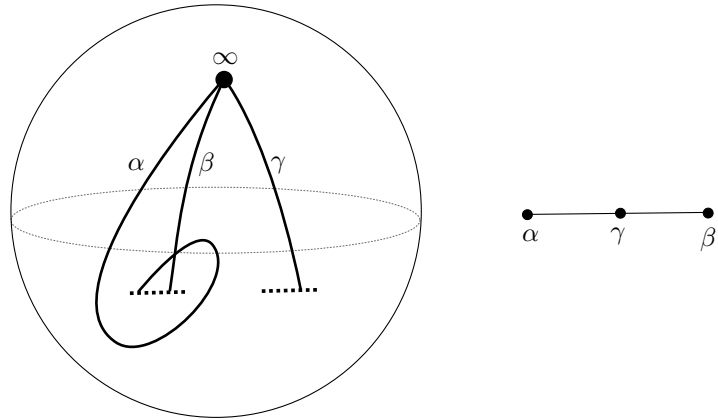}
\caption{Exemple de trois rayons représentés sur la sphère et du sous-graphe du graphe des rayons associé : $d(\alpha,\beta)=2$ et $d(\alpha,\gamma)=d(\beta,\gamma)=1$.}
\label{figu:graphe-rayons}
\end{figure}

On montre ici les résultats suivants :

\begin{theo*}[\ref{theo-infini}]
Le diamètre du graphe des rayons est infini.
\end{theo*}

\begin{theo*}[\ref{theo-hyperbolique}]
Le graphe des rayons est hyperbolique au sens de Gromov.
\end{theo*}

\begin{theo*}[\ref{theo-halphak}] Il existe un élément $h\in \Gamma$ agissant par translation sur un axe géodésique du graphe des rayons.
\end{theo*}

Ces résultats nous permettent de voir $\Gamma$ comme agissant non trivialement sur un espace hyperbolique. On cherche ensuite à utiliser cette action pour construire des \og quasi-morphismes non triviaux\fg \ sur $\Gamma$.

\subsection{Quasi-morphismes et cohomologie bornée}

Un \emph{quasi-morphisme} sur un groupe $G$ est une application $q:G\rightarrow \R$ telle qu'il existe une constante $D_q$, appelée \emph{défaut} du quasi-morphisme $q$, vérifiant pour tous $a,b \in G$ l'inégalité : $$|q(ab)-q(a)-q(b)|\leq D_q.$$
Les premiers exemples de quasi-morphismes sont les morphismes et les fonctions bornées. Ce sont des quasi-morphismes \og triviaux \fg : on dit qu'un quasi-morphisme $q$ est \emph{non trivial} si le quasi-morphisme $\tilde q$ défini par $\tilde q(a) = \lim_{n\rightarrow \infty} {q(a^n)\over n}$ pour tout $a\in G$ n'est pas un morphisme.

L'espace des quasi-morphismes non triviaux sur un groupe $G$, que l'on notera $\tilde Q(G)$, coïncide avec le noyau du morphisme naturel envoyant le deuxième groupe de cohomologie bornée $H^2_b(G;\R)$ de $G$ dans le deuxième groupe de cohomologie $H^2(G;\R)$ de $G$ (voir par exemple \cite{Barge-Ghys, Ghys-Groups} pour des précisions sur la cohomologie bornée des groupes). L'étude de cet espace $\tilde Q(G)$ donne des informations sur le groupe $G$ : par exemple, on sait qu'il est trivial lorsque $G$ est moyennable (voir \cite{Gromov}), ou lorsque $G$ est un réseau cocompact irréductible d'un groupe de Lie semi-simple de rang supérieur strictement à $1$ (voir \cite{Burger-Monod}).
Dans \cite{Bestvina-Fujiwara}, Mladen Bestvina et Koji Fujiwara ont montré que l'espace des quasi-morphismes non triviaux sur un groupe modulaire d'une surface de type fini est de dimension infinie, ce qui a de nombreuses conséquences et implique notamment que pour de nombreuses classes de groupes $G$, tout morphisme de $G$ vers un groupe modulaire de surface de type fini se factorise par un groupe fini.
  Ces résultats, ainsi que les applications potentielles en dynamique, motivent la recherche de quasi-morphismes non triviaux sur le groupe $MCG(\R^2-Cantor)$ proposée par Calegari \cite{blog-Calegari}. On montre ici le résultat suivant :
\begin{theo*}[\ref{dim infinie}] L'espace $\tilde Q(\Gamma)$ des quasi-morphismes non triviaux sur $\Gamma$ est de dimension infinie.
\end{theo*}

Ce résultat implique en particulier que la \og longueur stable des commutateurs \fg \ est une quantité non bornée sur $\Gamma$.

\subsection{Longueur stable des commutateurs}

Si $G$ est un groupe, on note $[G,G]$ son groupe dérivé, c'est-à-dire le sous-groupe de $G$ engendré par les commutateurs (éléments qui s'écrivent sous la forme $[x,y]=xyx^{-1}y^{-1}$ avec $x,y\in G$). Pour tout $a \in [G,G]$, on note $cl(a)$ la \emph{longueur des commutateurs} de $a$, c'est-à-dire le plus petit nombre de commutateurs dont le produit est égal à $a$. On définit alors la \emph{longueur stable des commutateurs} de $a$ par :
$$scl(a):=\lim_{n\rightarrow +\infty} {cl(a^n)\over n}.$$
C'est en particulier une quantité invariante par conjugaison (voir \cite{SCL} pour des précisions sur la longueur stable des commutateurs). L'étude de cette quantité est reliée à celle des quasi-morphismes non triviaux par un théorème de dualité : Christophe Bavard a montré que l'espace des quasi-morphismes non triviaux sur un groupe $G$ est trivial si et seulement si tous les éléments de $[G,G]$ sont de $scl$ nulle (voir \cite{CB}).

Dans le cas du groupe $\Gamma$, Calegari a montré dans \cite{blog-Calegari} que si $g\in \Gamma$ a une orbite bornée sur le graphe des rayons, alors $scl(g)=0$. Cette propriété rend encore plus surprenante l'existence d'un espace de quasi-morphismes non triviaux de dimension infinie sur $\Gamma$. De plus, elle distingue $\Gamma$ des groupes modulaires des surfaces de type fini, dont certains éléments ont une orbite bornée sur le complexe des courbes et une $scl$ non nulle : en effet, Endo-Kotschick \cite{Endo-Kotschick} et Korkmaz \cite{Korkmaz-scl} ont montré que les twists de Dehn sont de $scl$ strictement positive. Dans le cas des surfaces de type fini, on sait maintenant caractériser précisément en termes de la décomposition de Nielsen-Thurston les éléments de $scl$ nulle (voir Bestvina-Bromberg-Fujiwara \cite{Bestvina-Bromberg-Fujiwara}). Dans le cas de $\Gamma$, on peut s'interroger sur une éventuelle réciproque à la proposition de Calegari : est-ce que tous les éléments de $scl$ nulle ont une orbite bornée sur le graphe des rayons ? On exhibe ici un élément hyperbolique de $\Gamma$ de $scl$ nulle (proposition \ref{ex-scl_nulle-hyp}), montrant ainsi qu'une éventuelle caractérisation des éléments de $scl$ nulle serait plus fine que la classification entre éléments ayant ou non une orbite bornée.

\subsection{Stratégies de preuves}
  
  Dans la section \ref{section1}, on construit une suite de rayons  $(\alpha_k)_k$ qui est non bornée dans le graphe des rayons, montrant ainsi que le graphe des rayons est de diamètre infini.

\pgfdeclareimage[interpolate=true,height=5cm]{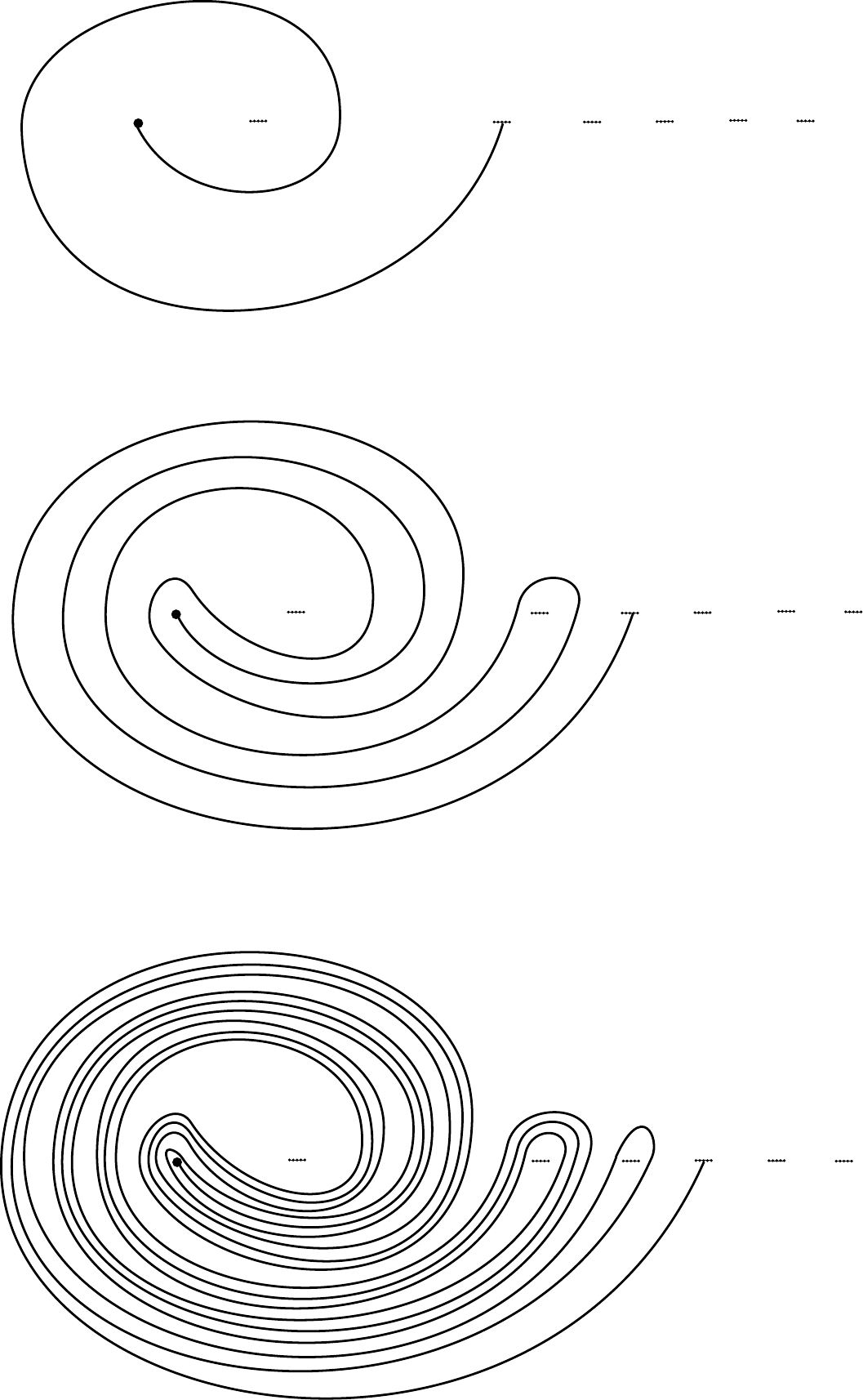}{tube-intro}
\begin{figure}[!h]
\labellist
\pinlabel $\infty$ at 49 458
\pinlabel $a_1$ at 176 416
\pinlabel $a_2$ at 212 227
\pinlabel $a_3$ at 223 24
\endlabellist
\centering
\includegraphics[scale=0.5]{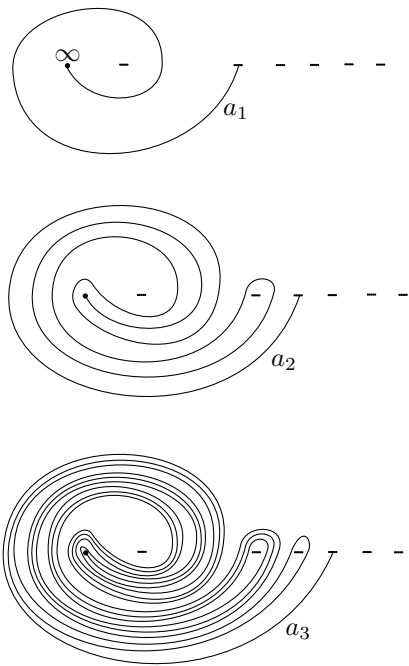}
\caption{Construction de $a_2$ à partir de $a_1$ et de $a_3$ à partir de $a_2$.}
\label{figu:tube}
\end{figure}

Cette suite est contruite par récurrence à partir de l'idée suivante : si l'on considère un arc $a_1$ représentant un rayon et un arc $a_2$ formant un \og tube \fg \ dans un petit voisinage autour de $a_1$ (comme sur la figure \ref{figu:tube}), tout arc disjoint de $a_2$ et représentant un rayon doit commencer en l'infini et finir en un point de l'ensemble de Cantor sans traverser $a_2$. Un tel arc doit alors \og suivre le parcours de l'arc $a_1$ \fg \ avant de pouvoir éventuellement s'échapper du tube formé par $a_2$ et rejoindre un point de l'ensemble de Cantor. Si maintenant $a_3$ est un arc représentant un rayon et qui forme un tube dans un voisinage autour de $a_2$ (voir la figure \ref{figu:tube}), le même phénomène se produit : tout arc disjoint de $a_3$ doit \og suivre le parcours de l'arc $a_2$ \fg \ avant de pouvoir s'échapper du tube formé par $a_3$ pour rejoindre un point de l'ensemble de Cantor. Ainsi, dans le graphe des rayons, tout rayon à distance $1$ du rayon représenté par $a_3$ commence par suivre le parcours de $a_2$, ce qui force tout rayon à distance $2$ de $a_2$ à suivre le parcours de $a_1$ : si $\beta$ est par exemple le rayon représenté par un arc qui joint l'infini au point d'attachement de $a_1$ en restant dans l'hémisphère nord, alors $\beta$ est à distance supérieure à $3$ de $a_3$. En effet, tout arc qui commence par parcourir $a_2$ ou $a_1$ n'est pas homotopiquement disjoint de $\beta$, donc tous les représentants des rayons à distance $1$ ou $2$ du rayon représenté par $a_3$ intersectent tout arc homotope à $\beta$. 

On choisit ensuite $a_4$ qui dessine un tube autour de $a_3$ : tout rayon à distance $1$ du rayon représenté par $a_4$ commence par suivre le parcours de $a_3$ ; ce qui implique que tout rayon à distance $2$ de $a_4$ commence par suivre le parcours de $a_2$ ; ce qui implique que tout rayon à distance $3$ de $a_4$ commence par suivre le parcours de $a_1$ ; ce qui implique que le rayon représenté par $a_4$ est à distance supérieure à $4$ du rayon $\beta$.

On peut continuer ainsi en choisissant $a_5$ qui forme un tube autour de $a_4$, etc. Pour tout $k$, on obtient un rayon $\alpha_k$ représenté par $a_k$ et tel que tout arc représentant un rayon à distance strictement inférieure à $k$ de $\alpha_k$ commence par suivre le parcours de $a_1$, coupant ainsi $\beta$. Pour rendre cette idée rigoureuse, on définit dans la section \ref{section1} un codage pour certains rayons, puis la suite $(\alpha_k)_{k\in \N}$ de rayons représentant les \og tubes\fg \ voulus. On montre grâce au codage que cette suite et non bornée dans le graphe des rayons (théorème \ref{theo-infini}), et qu'elle forme un demi-axe géodésique dans ce graphe (proposition \ref{alpha_k geod}).\\

Dans la section \ref{section2}, on montre que le graphe des rayons est hyperbolique au sens de Gromov (théorème \ref{theo-hyperbolique}). On définit pour cela un graphe annexe $X_\infty$ dont les sommets sont les classes d'homotopies de lacets simples de $\Sph^2 - K$ basés en l'infini, et dont les arêtes sont les paires de tels lacets ayant des représentants disjoints. On montre que ce graphe est hyperbolique en adaptant la preuve de l'uniforme hyperbolicité des complexes des arcs par les chemins \og unicornes\fg \ de Sebastian Hensel, Piotr Przytycki et Richard Webb \cite{HPW}. On montre ensuite que ce graphe $X_\infty$ est quasi-isométrique au graphe des rayons, ce qui permet d'établir l'hyperbolicité de ce dernier. On définit pour cela une application $f$ entre le graphe des rayons $X_r$ et le graphe hyperbolique $X_\infty$, qui à tout rayon $x$ de $X_r$ associe un élément $\hat x$ de $X_\infty$ tel que $x$ et $\hat x$ ont des représentants disjoints, puis on montre que cette application est une quasi-isométrie.\\

Dans la section \ref{section-qm}, on utilise à nouveau la suite de rayons $(\alpha_k)_k$ construite dans la section \ref{section1}, qui définit un axe géodésique du graphe des rayons. On exhibe un élément $h \in \Gamma$ qui agit par translation sur cet axe (théorème \ref{theo-halphak}). L'élément $h$ est un élément pouvant être représenté par la tresse de la figure \ref{figu-h}. Les points noirs représentent l'ensemble de Cantor $K$, et chaque brin transporte tous les points du sous-ensemble de Cantor correspondant. On montre que pour tout $k\in \N$, $h(\alpha_k)=\alpha_{k+1}$.

\begin{figure}[!h]
\labellist
\small\hair 2pt
\pinlabel $\infty$ at 211 204
\pinlabel $h$ at 465 110
\endlabellist
\centering
\vspace{0.2cm}
\includegraphics[scale=0.5]{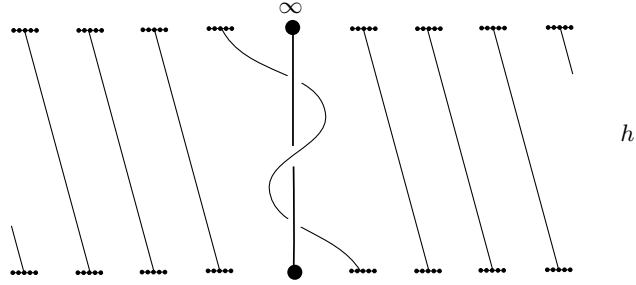}
\caption{Représentation de l'élément $h \in \Gamma$.}
\label{figu-h}
\end{figure}

On cherche ensuite à construire des quasi-morphismes non triviaux sur $\Gamma$. Dans \cite{Fujiwara}, Koji Fujiwara définit les quasi-morphismes de comptage sur les groupes agissant sur des espaces hyperboliques, généralisant la construction de Brooks \cite{Brooks} sur les groupes libres. Dans le cas des groupes modulaires des surfaces compactes de type fini, Mladen Bestvina et Koji Fujiwara utilisent cette construction pour montrer que l'espace des quasi-morphismes non triviaux est de dimension infinie (voir \cite{Bestvina-Fujiwara}). L'espace hyperbolique considéré est alors le complexe des courbes de la surface, sur lequel le groupe modulaire de la surface considérée agit \og faiblement proprement discontinûment\fg, propriété qui garantit en particulier la non-trivialité de certains quasi-morphismes obtenus par la construction de Fujiwara.

Comme on sait que $\Gamma$ agit sur un espace hyperbolique (le graphe des rayons), la construction \cite{Fujiwara} de Fujiwara nous donne des quasi-morphismes sur $\Gamma$. On cherche alors à montrer que certains de ces quasi-morphismes sont non triviaux. Malheureusement, l'action de $\Gamma$ sur le graphe des rayons n'est pas  \og faiblement proprement discontinue \fg \ (voir l'énoncé au début de la section \ref{rq-WPD}). On peut néanmoins définir un \og nombre d'intersections positives\fg, qui nous permet de montrer que l'axe $(\alpha_k)_k$ est \og non retournable \fg (proposition \ref{copies}). Cette propriété généralise le fait pour $h$ de ne pas être conjugué à son inverse. Plus précisément, on montre que pour tout segment orienté suffisamment long de l'axe $(\alpha_k)_k$, si un élément de $\Gamma$ envoie ce segment dans un voisinage \og proche \fg \ de l'axe $(\alpha_k)_k$, alors l'image du segment est orientée dans le même sens que le segment d'origine. Cette propriété de l'axe $(\alpha_k)_k$ ainsi que l'action de $h$ sur cet axe permettent de construire un quasi-morphisme non trivial explicite (proposition \ref{prop-qm}).

On utilise ensuite encore une fois l'élement $h \in \Gamma$, ainsi qu'un conjugué de son inverse, pour montrer grâce à un autre théorème de Bestvina et Fujiwara \cite{Bestvina-Fujiwara} et à la propriété \ref{copies} de non retournement que l'espace $\tilde Q(\Gamma)$ des quasi-morphismes non triviaux sur $\Gamma$ est de dimension infinie (théorème \ref{dim infinie}).

\subsection{Remerciements.}

Je remercie mon directeur de thèse, Frédéric Le Roux, pour sa grande disponibilité, ses nombreux conseils et ses relectures minutieuses des différentes versions de ce texte. Merci à Danny Calegari de l'intérêt qu'il a porté à ce travail, et de m'avoir suggéré d'ajouter un exemple d'élément hyperbolique de $scl$ nulle, en plus des questions posées sur son blog... Merci également à Nicolas Bergeron pour ses explications autour des surfaces hyperboliques.

\section{Première étude du graphe des rayons : diamètre infini et demi-axe géodésique}\label{section1}

On cherche ici à montrer que le graphe des rayons est de diamètre infini. Dans ce but, on va construire une suite de rayons $(\alpha_n)_{n\geq0}$ et montrer qu'elle n'est pas bornée dans le graphe des rayons. On code certains rayons par des suites de segments, pour pouvoir les manipuler plus facilement dans les preuves. On définit à partir de ce codage la suite de rayons $(\alpha_n)_n$ qui nous intéresse. On montre enfin que cette suite n'est pas bornée dans le graphe des rayons, et qu'elle définit un demi-axe géodésique. Les résultats montrés autour de cette suite nous seront à nouveau utiles dans la section \ref{section-qm}.

\subsection{Préliminaires}

On utilisera dans toute la suite les notations, propositions, et le vocabulaire suivants.

\subsubsection*{Ensemble de Cantor $K$}
On note $K$ un ensemble de Cantor plongé dans $\Sph^2$, et on choisit $\infty$ un point de $\Sph^2 - K$. On identifie $\R^2 -K$ et $\Sph^2-(K\cup \{\infty\})$. Si $K'$ est un autre ensemble de Cantor plongé dans $\Sph^2$ et $\infty'$ un point de $\Sph^2 - K'$, alors il existe un homéomorphisme de $\Sph^2$ qui envoie $K'$ sur $K$ et $\infty'$ sur $\infty$ (voir par exemple l'appendice $A$ de \cite{Beguin-Crovisier-FLR_Cantor}).

\subsubsection*{Arcs, homotopies et isotopies}
Soit ${a}:[0,1] \rightarrow \Sph^2$ une application continue telle que $\{ {a}(0)\}$ et $\{{a}(1)\}$ sont inclus dans  $K\cup \{\infty\}$ et telle que ${a}(]0,1[)$ est inclus dans $\Sph^2-(K\cup \{\infty\})$. On appellera  \emph{arc} cette application ${a}$, que l'on confondra parfois avec l'image de $]0,1[$ par ${a}$. Si de plus l'application ${a}$ est injective, on dira que ${a}$ est un  \emph{arc simple} de $\Sph^2-(K \cup \{\infty\})$.\\
On dira que deux arcs ${a}$ et ${b}$ de $\Sph^2-(K \cup \{\infty\})$ sont \emph{homotopes} s'il existe une application continue $H:[0,1]\times [0,1] \rightarrow \Sph^2$ telle que :
\begin{itemize}
\item $H(0,\cdot)={a}(\cdot)$ et $H(1,\cdot)={b}(\cdot)$.
\item $H(\cdot,0)$ et $H(\cdot,1)$ sont constantes (les extrémités sont fixes).
\item $H(t,s) \in \Sph^2-(K \cup \{\infty\})$ pour tous $(t,s) \in [0,1]\times ]0,1[$.
\end{itemize}
Si ${a}$ et ${b}$ sont simples, homotopes, et si de plus il existe une homotopie $H$ telle que pour tout $t\in[0,1]$, $H(t,\cdot)$ est un arc simple, alors on dira que ${a}$ et ${b}$ sont \emph{isotopes}. David Epstein a montré que sur une surface, deux arcs homotopes sont isotopes (voir (\cite{Epstein}). Dans ce texte, on confondra isotopie et homotopie sur les surfaces.

On dira que deux classes d'isotopies d'arcs $\alpha$ et $\beta$ sont \emph{homotopiquement disjointes} s'il existe des représentants $a$ de $\alpha$ et $b$ de $\beta$ tels que $a(]0,1[)$ et $b(]0,1[)$ sont disjoints. On dira que deux arcs $a$ et $b$ sont \emph{homotopiquement disjoints} s'ils représentent deux classes d'isotopies homotopiquement disjointes. Un \emph{bigone} entre deux arcs $a$ et $b$ est une composante connexe du complémentaire de $a\cup b$ dans $\Sph^2-(K\cup \{\infty\})$ homéomorphe à un disque et dont le bord est la réunion d'un sous-arc de $a$ et d'un sous-arc de $b$. On dira que deux arcs propres $a$ et $b$ sont \emph{en position d'intersection minimale} si toutes leurs intersections sont transverses et s'il n'y a aucun bigone entre $a$ et $b$.

\subsubsection*{Graphe des rayons}
\begin{definition}
Un  \emph{rayon} est une classe d'isotopie d'arcs simples $\alpha$ ayant pour extrémités $\alpha(0)=\infty$ et $\alpha(1)\in K$. On appellera  \emph{point d'attachement du rayon} le point $\{\alpha(1) \}$.
\end{definition}

\begin{definition}
Le  \emph{graphe des rayons}, noté $X_r$, est le graphe défini comme suit (voir \cite{blog-Calegari}) :
\begin{itemize}
\item Les sommets sont les rayons définis précédemment.
\item Deux sommets sont reliés par une arête si et seulement si ils sont homotopiquement disjoints.
\end{itemize}
\end{definition}

\subsubsection*{Préliminaires sur les classes d'isotopies de courbes}

On utilisera à plusieurs reprises les résultats suivants, adaptés de \cite{Casson-Bleiler}, \cite{Handel} et \cite{Matsumoto}. On munit  $\Sph^2 - (K\cup \{ \infty\})$ d'une métrique hyperbolique complète de première espèce. Son revêtement universel est le plan hyperbolique $\mathbb{H}^2$.

\begin{prop} \label{prop 3.5}
Soient $\mathcal A$ et $\mathcal B$ deux familles localement finies d'arcs simples de $\Sph^2-(K\cup \{\infty\})$ telles que tous les éléments de $\mathcal A$ (respectivement $\mathcal B$) sont deux à deux homotopiquement disjoints. On suppose que pour tous $a \in \mathcal A$ et $b \in \mathcal B$, $a$ et $b$ sont en position d'intersection minimale.\\
Alors il existe un homéomorphisme isotope à l'identité par une isotopie qui fixe $K\cup \{ \infty \}$ en tout temps et telle que pour tous $a \in \mathcal A$ et $b \in \mathcal B$, $h(a)$ et $h(b)$ sont géodésiques.
\end{prop}

\begin{prop} \label{prop-hyp}
Soit ${a}$ et ${b}$ deux arcs de $\Sph^2 - (K \cup \{\infty\})$. Si $\tilde {a}$ est un relevé de $a$ au revêtement universel, alors il existe deux points $p^-$ et $p^+$ du bord $\partial \mathbb{H}^2$ du revêtement universel $\mathbb{H}^2$ tels que $\tilde a(t)$ tend vers $p^-$, respectivement $p^+$ lorsque $t$ tend vers $0$, respectivement $1$. On appelle extrémités de $\tilde a$ ces deux points. Si $\tilde {a}$ et $\tilde {b}$ sont deux relevés respectifs de ${a}$ et ${b}$ au revêtement universel qui ont les mêmes extrémités au bord, alors ${a}$ et ${b}$ sont isotopes dans  $\Sph^2 - (K \cup \{\infty\})$.
\end{prop}

\subsection{Codage de certains rayons}\label{partie codage}

\subsubsection*{\'Equateur}

On choisit à l'aide de la proposition \ref{prop 3.5} un cercle topologique $\E$ de $\Sph^2$ contenant $K \cup \{ \infty\}$ et tel que tous les segments de $\E - (K\cup \{\infty\})$ sont géodésiques. On appellera  \emph{équateur} ce cercle. On choisit une orientation sur l'équateur, et on appelle  \emph{hémisphère nord} le cercle topologique situé à sa gauche, et  \emph{hémisphère sud} celui situé à sa droite.

\subsubsection*{Choix de segments de $\E$}

Comme sur la figure figure \ref{figu:segments}, on choisit un point $p$ de $\E - \{\infty\}$ tel que les deux composantes connexes de $\E - \{\infty,p\}$ contiennent chacune des points de $K$. On choisit ensuite une suite $(p_n)_{n\in \N}$ de points de $K$ sur la composante connexe de $\E - \{\infty,p\}$ située à droite de $\infty$, telle que ${p}_0$ est le premier point de $K$ à droite de $\infty$ sur $\E$ et $p_{n+1}$ est à droite de $p_n$ pour tout $n\in \N$. On choisit de même une suite $(p_n)_{n<0}$ sur la composante connexe de $\E - \{\infty,p\}$ située à gauche de $\infty$, telle que  ${p}_{-1}$ est le premier point à gauche de $\infty$ et telle que $p_{n-1}$ est à gauche de $p_n$ pour tout $n <0$. On note $s_0$ la composante connexe de $\E - (K\cup \{\infty\})$ entre $\infty$ et $p_0$ et $s_{-1}$ celle entre $\infty $ et $p_{-1}$. On choisit pour tout $n>0$ une composante connexe $s_n$ de $\E - K$ entre $p_{n-1}$ et $p_n$, et pour tout $n<-1$  une composante connexe $s_n$ de $\E - K$ entre $p_{n}$ et $p_{n+1}$. On note $S$ l'ensemble des segments topologiques $\{s_n\}_{n\in \Z}$, et $\textbf{S}$ leur union $\bigcup_{n\in \Z} s_n$.

\pgfdeclareimage[interpolate=true,height=5cm]{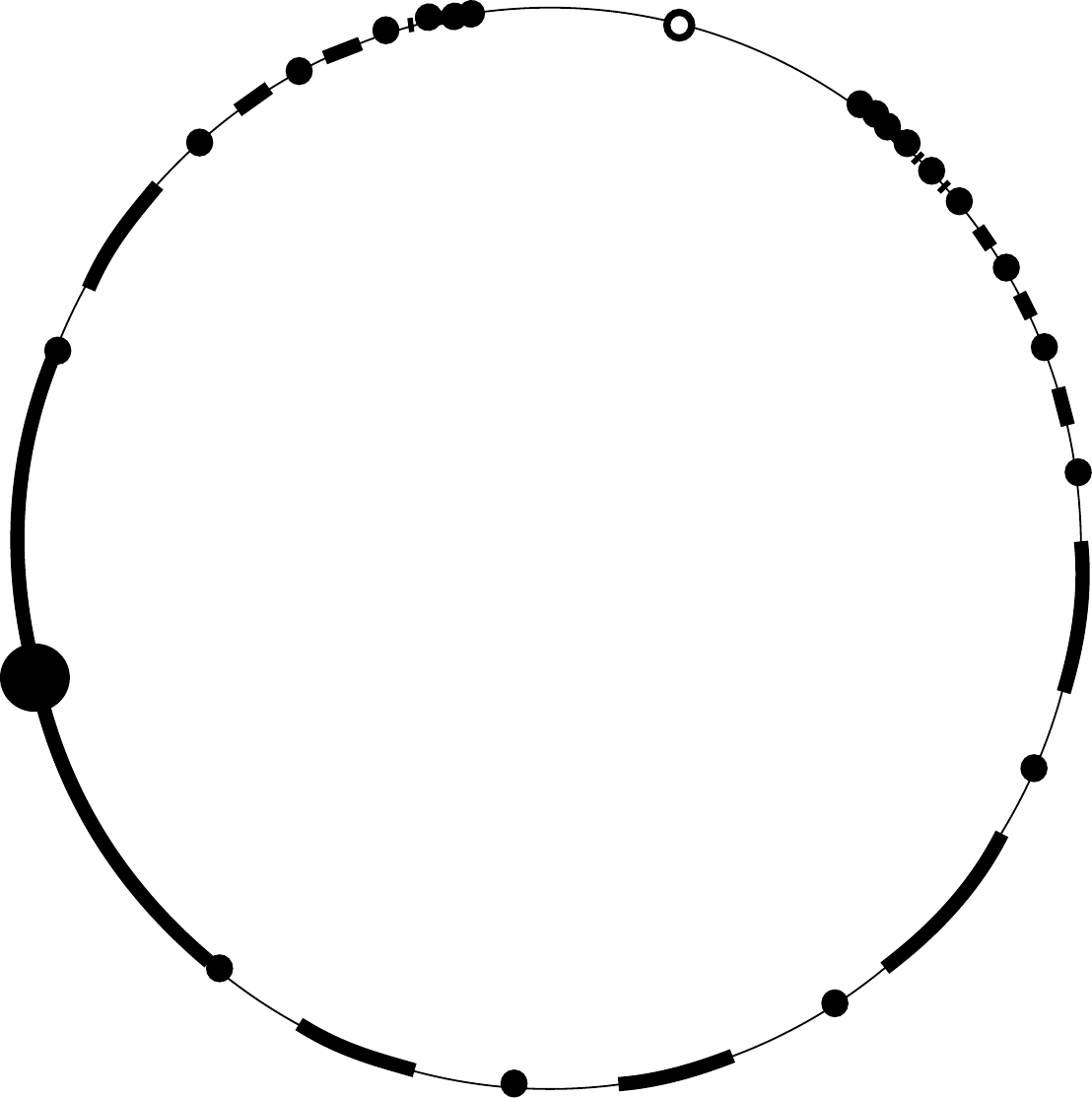}{segments}
\begin{figure}[!h]
\labellist
\small\hair 2pt
\pinlabel $\infty$ at -13 122
\pinlabel $s_{-1}$ at 26 174
\pinlabel $s_0$ at 39 80
\pinlabel ${p}_{-1}$ at -5 210
\pinlabel ${p}_0$ at 57 24
\pinlabel $s_1$ at 104 21
\pinlabel $s_2$ at 187 13
\pinlabel $s_3$ at 263 59
\pinlabel $s_4$ at 298 140
\pinlabel ${p}_1$ at 150 -12
\pinlabel ${p}_2$ at 251 11
\pinlabel ${p}_3$ at 316 87
\pinlabel ${p}_4$ at 332 186
\pinlabel $p$ at 201 324
\pinlabel \textit{Hémisphère Nord} at 127 248
\pinlabel \textit{Hémisphère Sud} at -77 248
\endlabellist
\centering
\includegraphics[scale=0.5]{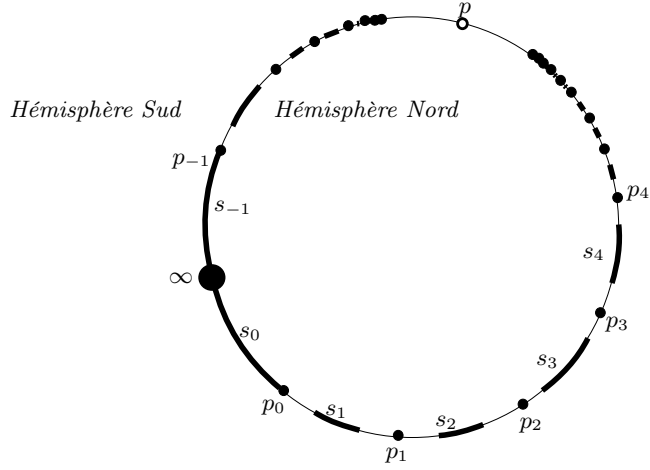}
\vspace{0.5cm}
\caption{Choix d'un équateur, d'un point $p$, d'une suite de points de $K$ et d'un ensemble de segments.}
\label{figu:segments}
\end{figure}

\subsubsection*{Suite associée}

Si $\alpha$ est une classe d'isotopie d'arcs de $\Sph^2-(K \cup \{\infty\})$, on notera $\alpha_\#$ l'unique arc géodésique représentant $\alpha$ dans $\Sph^2- (K \cup \{\infty\})$. On note $X'_S$ l'ensemble des classes d'isotopies d'arcs $\alpha$ de $\Sph^2-(K \cup \{\infty\})$ joignant l'infini et un point de l'ensemble de Cantor $K$ (éventuellement avec auto-intersection) tels que :
\begin{enumerate}
\item $\E \cap \alpha_\# \subset \textbf{S}$.
\item La composante connexe de $\alpha_\#  - \E$ qui part de $\infty$ est incluse dans l'hémisphère sud.
\item $\E \cap \alpha_\#$ est un ensemble fini.
\end{enumerate}
On note $X_S$ le sous-ensemble de $X'_S$ composé des classes d'isotopies d'arcs simples (c'est-à-dire l'ensemble des rayons vérifiant les trois propriétés précédentes).

Soit $\alpha \in X'_S$. On peut associer à $\alpha$ une suite de segments de la manière suivante :
on parcourt $\alpha_\#$ depuis $\infty$ et jusqu'à son point d'attachement, et on note $u_1$ le premier segment de $S$ intersecté par $\alpha_\#$, $u_2$ le second, ..., et $u_k$ le k-ième pour tout $k$, jusqu'à avoir atteint le point d'attachement.
On note $\mathring u(\alpha)$ cette suite de segments (finie), et $u(\alpha)$ la suite $\mathring u(\alpha)$ à laquelle on ajoute le point d'attachement, et que l'on appelle  \emph{suite complète associée à $\alpha$} (voir la figure \ref{exemple-suite} pour un exemple). Comme la géodésique $\alpha_\#$ est unique dans la classe d'isotopie $\alpha$, la suite de segment associée à $\alpha$ est bien définie. De façon générale, on appellera  \emph{suite complète} la donnée d'une suite finie de segments et d'un point de $K$.

\pgfdeclareimage[interpolate=true,height=5cm]{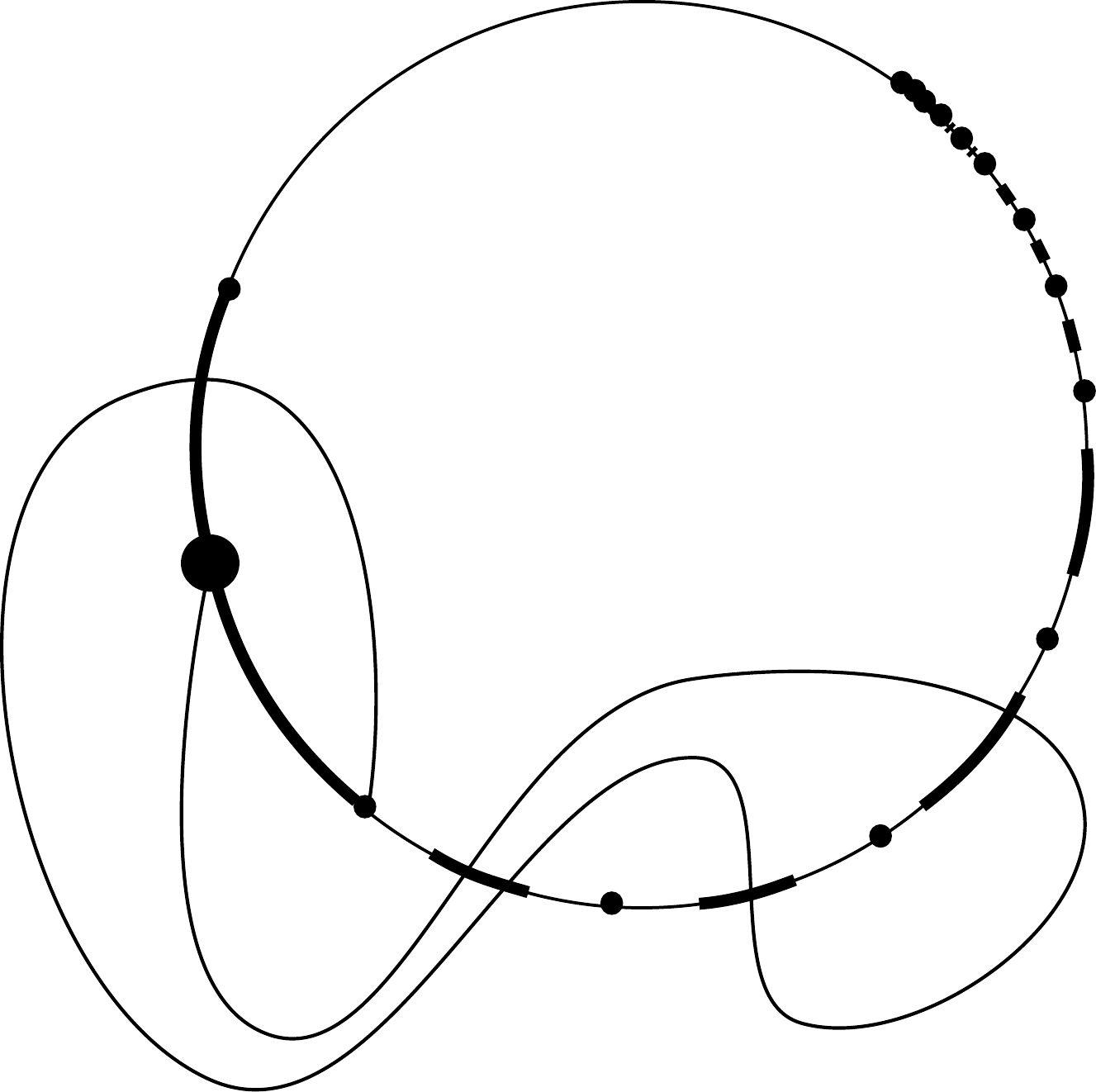}{exemple-suite}
\begin{figure}[!h]
\labellist
\small\hair 2pt
\pinlabel $\infty$ at 51 182
\pinlabel ${p}_0$ at 116 84
\pinlabel ${p}_1$ at 213 50
\pinlabel ${p}_2$ at 310 72
\pinlabel ${p}_3$ at 384 157
\pinlabel ${p}_4$ at 394 247
\pinlabel $\gamma$ at 3 50
\pinlabel \textit{Nord} at 154 332
\pinlabel \textit{Sud} at 64 332
\endlabellist
\centering
\includegraphics[scale=0.5]{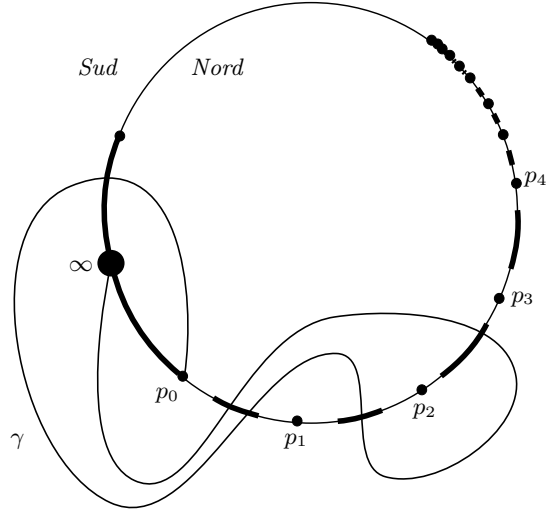}
\vspace{0.5cm}
\caption{Exemple d'un rayon $\gamma \in X_S$ : ici, le point d'attachement est ${p}_0$, la suite complète de segments associée est $u(\gamma)=s_1s_3s_2s_1s_{-1}({p}_0)$, et on a $\mathring u(\gamma)=s_1s_3s_2s_1s_{-1}$.}
\label{exemple-suite}
\end{figure}

\begin{lemme}
\`A chaque suite complète correspond une unique classe d'isotopie d'arcs de $X'_S$ (éventuellement avec auto-intersections) entre l'infini et un point de $K$. En particulier, si deux rayons de $X_S$ ont la même suite complète associée, alors ils sont égaux.
\end{lemme}

\begin{proof} Soient $\alpha$ et $\beta$ deux arcs ayant la même suite complète associée, disons $u_0...u_n({p}_j)$. Au revêtement universel, on choisit un relevé $\tilde \infty$ de $\infty$ (sur le bord du disque hyperbolique). On peut voir ce point $\tilde \infty$ comme la limite au bord d'un relevé quelconque de $\alpha$. On relève ensuite $\beta$ à partir de ce point. Le revêtement universel est pavé par des demi-domaines fondamentaux correspondant aux relevés d'un hémisphère : chaque demi-domaine fondamental a pour bord un relevé de l'équateur. On commence à relever $\alpha$ et $\beta$ à partir de $\tilde \infty$ dans un même demi-domaine fondamental $F_0$ (correspondant à un relevé de l'hémisphère sud). On définit $(F_i)_{0\leq i\leq n}$ comme la suite des relevés alternativement de l'hémisphère nord et sud, traversés par $\tilde \alpha_\#$. On remarque que $(F_i)_i$ est entièrement déterminée par le codage : on sort de $F_0$ pour arriver dans un relevé $F_1$ de l'hémisphère nord en traversant le seul relevé de $u_0$ qui borde $F_0$. On continue ainsi jusqu'au demi-domaine $F_n$, qui a un seul relevé $\tilde {p}_j$ de ${p}_j$ dans son bord. Ainsi les deux relevés $\tilde \alpha$ et $\tilde \beta$ de $\alpha$ et $\beta$ ont mêmes extrémités, donc $\alpha$ et $\beta$ sont isotopes dans $\Sph^2 - (K \cup \{\infty\})$ (d'après la proposition \ref{prop-hyp}). \end{proof}

\`A partir de maintenant on ne fera plus de différence explicite entre une classe d'isotopie d'arcs de $X'_S$ et sa suite complète associée.

\paragraph{Remarque :} Les suites complètes de segments correspondant à des rayons ne commencent jamais ni par $s_{-1}$, ni par $s_0$, et elles n'ont jamais plusieurs fois de suite le même segment (sinon il y a un bigone).

\subsection{Une suite de rayons particulière}
On construit ici une suite particulière de rayons, $(\alpha_k)_{k\in \N}$, dont les propriétés nous seront utiles pour toute la suite. \\

Si $u=u_0u_1...u_n ({p}_j)$ est une suite complète de segments, on rappelle qu'on note $\mathring u = u_0 u_1 ...u_n$ la suite de segments sans le point d'attachement.
On notera alors $\mathring u^{-1}:=u_n...u_1 u_0$ la suite de segments inverse.

\pgfdeclareimage[interpolate=true,height=5cm]{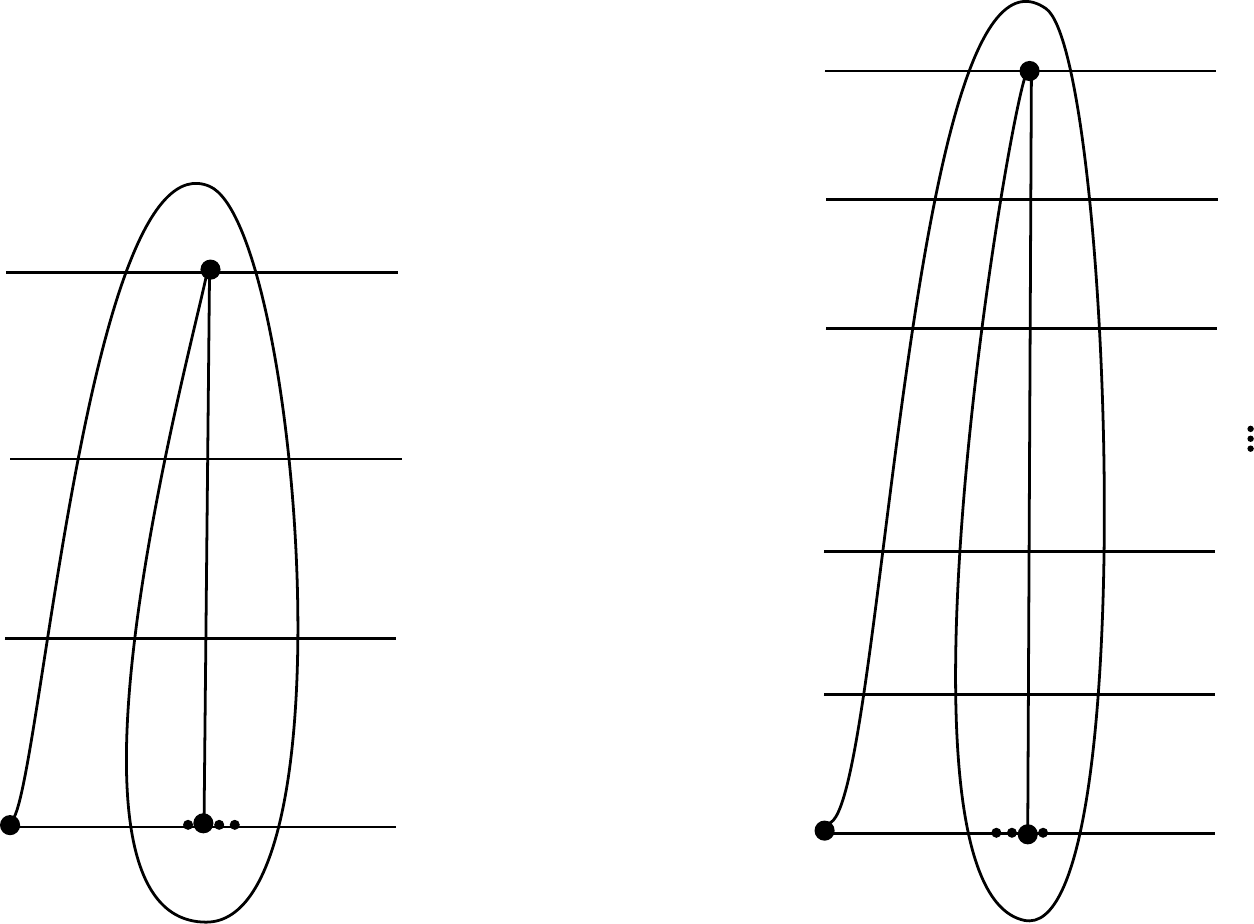}{definition-alpha}
\begin{figure}[!h]
\labellist
\small\hair 2pt
\pinlabel $\infty$ at 60 195
\pinlabel $\infty$ at 297 251
\pinlabel $s_{-1}$ at -11 187
\pinlabel $s_0$ at 123 187
\pinlabel $s_{-1}$ at 227 245
\pinlabel $s_0$ at 360 245
\pinlabel $s_{-1}$ at 360 171
\pinlabel $s_1$ at 123 133
\pinlabel $s_{-1}$ at 123 80
\pinlabel $s_1$ at 123 27
\pinlabel $s_{-1}$ at 360 65
\pinlabel $s_1$ at 360 208
\pinlabel $s_{2}$ at 25 32
\pinlabel $s_1$ at 360 106
\pinlabel $s_{k+1}$ at 260 30
\pinlabel $s_k$ at 360 26

\pinlabel ${p}_1$ at 58 18
\pinlabel ${p}_2$ at 3 18
\pinlabel ${p}_k$ at 296 17
\pinlabel ${p}_{k+1}$ at 240 17

\pinlabel \textit{Nord} at -15 214
\pinlabel \textit{Sud} at -15 163
\pinlabel \textit{Nord} at -15 109
\pinlabel \textit{Sud} at -15 58
\pinlabel \textit{Nord} at -15 7

\pinlabel \textit{Sud} at 230 226
\pinlabel \textit{Nord} at 230 189

\pinlabel \textit{Sud} at 230 46

\pinlabel $\alpha_1$ at 69 106
\pinlabel $\alpha_2$ at 66 215
\pinlabel $\alpha_k$ at 307 130
\pinlabel $\alpha_{k+1}$ at 304 270
\endlabellist
\centering
\vspace{0.3cm}
\includegraphics[scale=0.8]{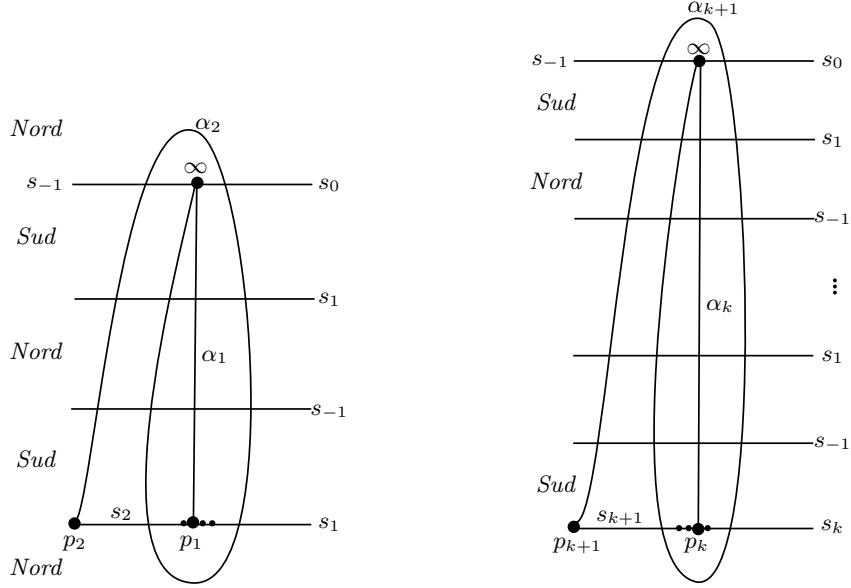}
\vspace{0.5cm}
\caption{Définition de $\alpha_2$ à partir de $\alpha_1$, puis de $\alpha_{k+1}$ à partir de $\alpha_k$ : représentation des intersections locales de ces rayons avec $\E$.}
\label{def-alpha}
\end{figure}

\begin{definition} On définit la suite $(\alpha_k)_{k\geq 0}$ de rayons de la façon suivante:
\begin{itemize}
\item $\alpha_0$ est la classe d'isotopie du segment $s_0$, avec pour extrémités $\infty$ et ${p}_0$.
\item $\alpha_1$ est le rayon codé par $s_1 s_{-1} ({p}_1)$ (voir figure \ref{alpha1-2}).
\item Pour tout $k\geq 1$, $\alpha_{k+1}$ est le rayon défini à partir de $\alpha_k$ comme sur la figure \ref{def-alpha} : on part de $\infty$, on longe $\alpha_{k\#}$  jusqu'à son point d'attachement ${p}_k$ dans un voisinage tubulaire de $\alpha_{k\#}$, on contourne ce point par la droite en traversant les segments voisins, c'est-à-dire en traversant d'abord $s_{k+1}$ puis $s_k$, on longe à nouveau $\alpha_{k\#}$ dans un voisinage tubulaire, on contourne $\infty$ en traversant $s_0$ puis $s_{-1}$, on longe une dernière fois $\alpha_{k\#}$ dans un voisinage tubulaire jusqu'à son point d'attachement et on va s'attacher au point ${p}_{k+1}$ sans traverser l'équateur.
\end{itemize}
\end{definition}

En termes de codage, on obtient les suites complètes suivantes :
\begin{itemize}
\item $\alpha_0 =s_0({p}_0)$.
\item $\alpha_1 = s_1 s_{-1} ({p}_1)$.
\item $\alpha_{k+1}=\mathring \alpha_k s_{k+1} s_k \mathring \alpha_k^{-1} s_0 s_{-1} \mathring \alpha_k ({p}_{k+1})$ pour tout $k \geq 1$.
\end{itemize}

\paragraph{Remarque :} Si l'on note $long(\alpha_k)$ le nombre de fois que $\alpha_{k\#}$ traverse un hémisphère, c'est-à-dire le nombre de composantes connexes de $\alpha_{k\#} - \E$, ou encore le nombre de copies de demi-domaines fondamentaux traversés par un relevé géodésique $\tilde \alpha_k$ au revêtement universel, alors $long(\alpha_k)$ est impair pour tout $k \geq 1$. En effet on a $long(\alpha_1)=3$ (voir figure \ref{def-alpha}) et par construction $long(\alpha_{k+1}) = 3 long(\alpha_k)+2$ donc $long(\alpha_{k+1})$ a la même parité que $long(\alpha_k)$. Ainsi on est sûr d'être dans la situation de la figure \ref{def-alpha}, à savoir que le dernier hémisphère traversé par $\alpha_k$ est l'hémisphère sud, donc ${p}_{k+1}$ est toujours à gauche de ${p}_{k}$ dans la représentation choisie (figure \ref{def-alpha}), et lorsque $\alpha_{k+1}$ contourne $\infty$, ce rayon traverse d'abord $s_0$ puis $s_{-1}$ pour éviter toute auto-intersection.

\pgfdeclareimage[interpolate=true,height=5cm]{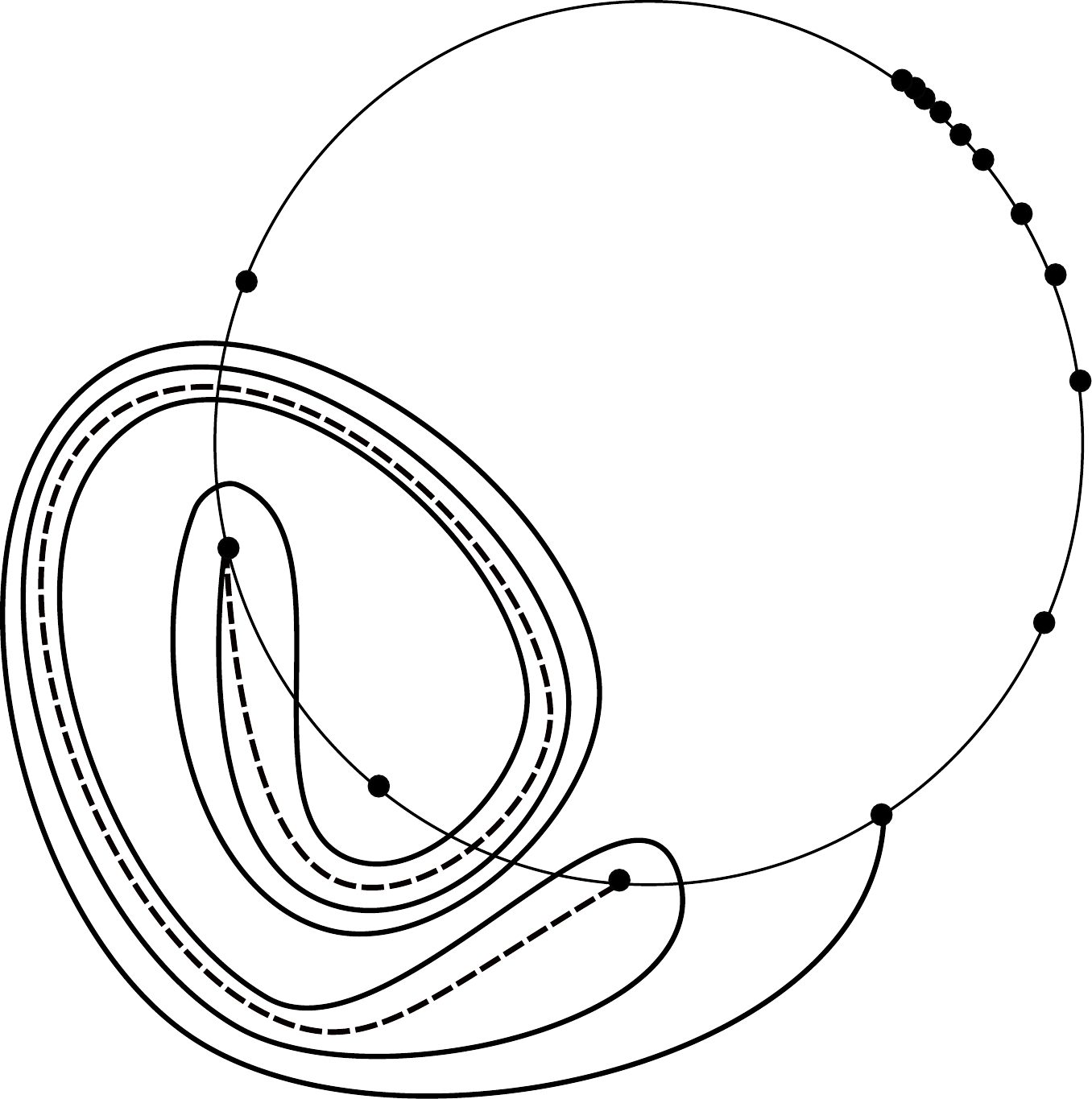}{alpha1-2}
\begin{figure}[!h]
\labellist
\small\hair 2pt
\pinlabel $\infty$ at 82 206
\pinlabel ${p}_0$ at 125 102
\pinlabel ${p}_1$ at 225 62
\pinlabel ${p}_2$ at 329 92
\pinlabel ${p}_3$ at 389 164
\pinlabel ${p}_4$ at 404 257
\pinlabel \textit{Nord} at 154 332
\pinlabel \textit{Sud} at 64 332
\endlabellist
\centering
\includegraphics[scale=0.5]{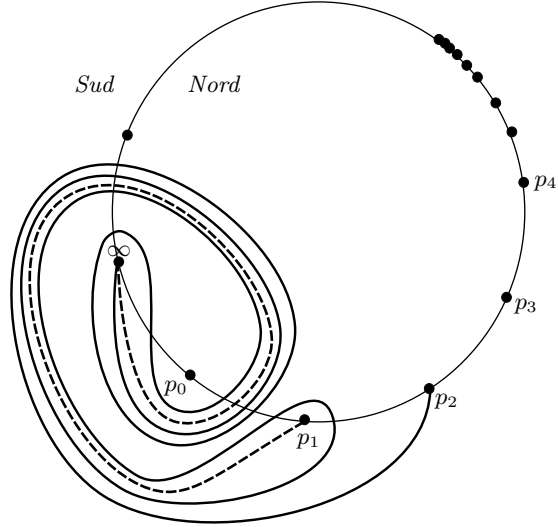}
\vspace{0.5cm}
\caption{Représentation sur la sphère de $\alpha_1= s_1 s_{-1} ({p}_1)$ en pointillés et de $\alpha_2=s_1s_{-1} s_2s_1s_{-1} s_1s_0s_{-1} s_1s_{-1}({p}_2)$ en trait plein.}
\label{alpha1-2}
\end{figure}

\subsection{Diamètre infini et demi-axe géodésique} \label{def de A}

Soit $\beta$ un rayon et $\mathring u = u_0 u_1...u_n$ une suite de segments. On dira que $\beta$  \emph{commence par $\mathring u$} si la première composante connexe de $\beta_ \# - \E$ est dans l'hémisphère sud et si les premières intersections de $\beta_\#$ avec $\E$ sont, dans cet ordre, les segments $u_0,u_1,...,u_n$. En particulier, si $\beta \in X_S$, ceci revient à dire que $u(\beta)$ commence par $\mathring u$.

\begin{definition}
Soit $A:X_r \rightarrow \N$ l'application qui à toute classe d'isotopie de rayon $\gamma$ associe :
$$A(\gamma):=Max\{i\in \N  \emph{ tel que } \gamma  \emph{ commence par } \mathring \alpha_i\}.$$
\end{definition}

\noindent Comme $\mathring \alpha_0$ est la suite vide, $A$ est bien définie pour tout $\gamma \in X_r$. Le lemme suivant montre que l'application $A$ est $1$-lipschitzienne.
\begin{lemme} \label{lemme distance}
Soient $\beta$ et $\gamma$ deux rayons tels que $d(\gamma,\beta)=1$. Alors : $$|A(\gamma) - A(\beta)| \leq 1.$$
\end{lemme}

\begin{proof}
On pose $n:=A(\beta)$. On choisit des représentants géodésiques $\beta_\#$ de $\beta$ et $\gamma_\#$ de $\gamma$ (voir figure \ref{prop-alpha}). L'arc $\beta_\#$ commence par parcourir la courbe représentant $\mathring \alpha_n$ : en effet, il doit traverser les mêmes segments, dans le même ordre. Il existe un homéomorphisme fixant chaque point de $K$ et $\infty$, fixant globalement $\E$ et envoyant le début de $\beta_\#$, c'est-à-dire la composante de $\beta_\#$ entre $\infty$ et $s_{-1}$, sur le début de $\alpha_n$, c'est-à-dire  la composante de $\alpha_n$ entre $\infty$ et $s_{-1}$. Comme $\gamma$ est à distance $1$ de $\beta$, $\gamma_\#$ est disjoint de $\beta_\#$ et doit sortir de la zone grise, qui ne contient aucun point de $K$, pour s'accrocher à un point de $K$ sans couper $\beta_\#$, donc sans couper la courbe pleine sur la figure \ref{prop-alpha}. Ainsi $\gamma_\#$ doit commencer par parcourir une des deux flèches pointillées, ce qui revient exactement à dire que $\gamma$ commence par $\mathring \alpha_{n-1}$. On a donc $A(\gamma)\geq n-1$. Par symétrie, on a le résultat voulu. \end{proof}

\pgfdeclareimage[interpolate=true,height=5cm]{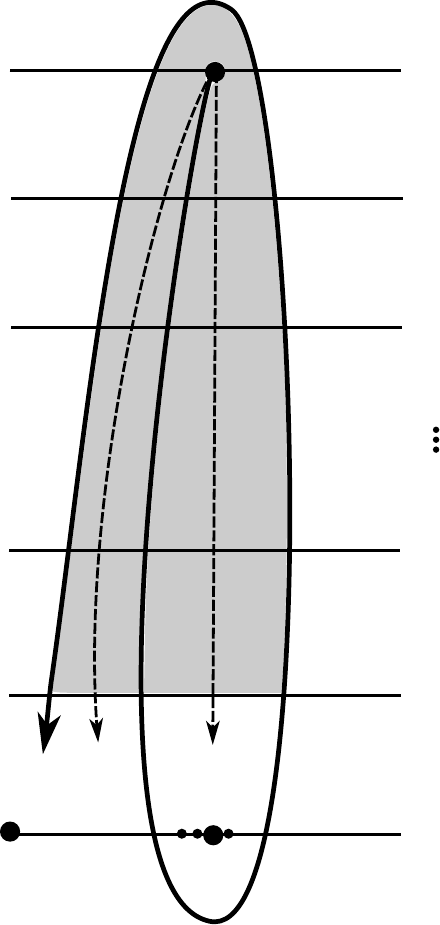}{prop-alpha}
\begin{figure}[!h]
\labellist
\small\hair 2pt
\pinlabel $\mathring \alpha_n$ at 90 219
\pinlabel $\infty$ at 62 253
\pinlabel $s_{-1}$ at 127 65
\pinlabel $s_{-1}$ at -10 245
\pinlabel $s_0$ at 127 245
\pinlabel $s_1$ at 127 210
\pinlabel ${p}_n$ at 2 15
\pinlabel ${p}_{n-1}$ at 70 15
\endlabellist
\centering
\includegraphics[scale=0.7]{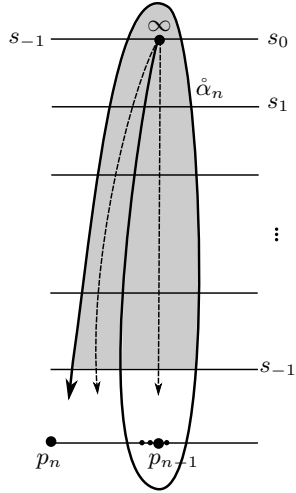}
\caption{Représentation des intersections locales de $\mathring \alpha_{n}$ avec $\E$. Par définition de $(\alpha_k)_k$, il n'y a aucun point de $K$ dans la zone grisée.}
\label{prop-alpha}
\end{figure}

\begin{corollaire} \label{dist}
Soient $\beta$ et $\gamma$ deux rayons quelconques de $X_r$. On a : $$|A(\beta)-A(\gamma)| \leq d(\beta,\gamma).$$
\end{corollaire}

\begin{proof}
On choisit une géodésique dans le graphe des rayons entre $\beta$ et $\gamma$ et par sous-additivité de la valeur absolue on en déduit le résultat grâce au lemme \ref{lemme distance}. \end{proof}

\noindent Cette inégalité nous permet de minorer certaines distances, et en particulier on en déduit le théorème suivant :

\begin{theo}\label{theo-infini}
Le diamètre du graphe des rayons est infini.
\end{theo}

\begin{proof} Par définition de $A$, on a $A(\alpha_0)=0$ et $A(\alpha_n)=n$ pour tout $n\in \N$. D'après le corollaire \ref{dist}, on a donc $d(\alpha_0,\alpha_n)\geq n$.
\end{proof}

\begin{prop} \label{alpha_k geod}
Le demi-axe $(\alpha_k)_{k\in \N}$ est géodésique.
\end{prop}

\begin{proof}
Par construction de la suite $(\alpha_k)_{k\in \N}$, on a $d(\alpha_k,\alpha_{k+1})=1$  pour tout $k\geq 0$. Par ailleurs $d(\alpha_k,\alpha_0)\geq k$ pour tout $k\geq 0$ (c'est une conséquence du corollaire \ref{dist}). Ainsi pour tout $k \geq 0$, on a $d(\alpha_k,\alpha_0)=k$.\end{proof}

\section{Hyperbolicité du graphe des rayons}\label{section2}

On dira qu'un espace métrique $X$ est \emph{géodésique} si entre deux points quelconques de $X$ il existe toujours au moins une géodésique, c'est-à-dire un chemin qui minimise la distance entre ces deux points. On rappelle la définition d'espace métrique hyperbolique au sens de Gromov. Pour plus de précisions sur les espaces hyperboliques, on pourra consulter par exemple \cite{Bridson-Haefliger}.

\begin{definition}[Espace hyperbolique]
On dira qu'un espace métrique géodésique $X$ est \emph{hyperbolique au sens de Gromov}, ou tout simplement \emph{hyperbolique}, s'il existe une constante $\delta\geq 0$ telle que pour tout triangle géodésique de $X$, chaque côté du triangle est inclus dans le $\delta$-voisinage des deux autres.
\end{definition}

On définit un graphe $X_\infty$ et on montre qu'il est hyperbolique par les mêmes arguments que ceux développés dans \cite{HPW} pour montrer l'hyperbolicité du graphe des arcs dans le cas des surfaces compactes à bord. On utilise ensuite cette hyperbolicité pour établir l'hyperbolicité du graphe des rayons. 

\subsection{Hyperbolicité du graphe des lacets simples basés en l'infini}

\subsubsection*{Graphe $X_\infty$ et chemins \og unicornes \fg}
On fixe $K$ un ensemble de Cantor de $\R^2$ et on compactifie $\R^2$ en ajoutant $\infty$, obtenant ainsi la sphère $\Sph^2$.
Un arc simple de $\Sph^2-K$ joignant l'infini à l'infini est dit \emph{essentiel} s'il ne borde pas un disque topologique, c'est-à-dire qu'il sépare la sphère en deux composantes dont chacune contient des points de $K$.

\begin{definition}
On construit le graphe $X_\infty$ comme suit :
\begin{itemize}
\item Les sommets sont les classes d'isotopies des arcs simples essentiels sur $\Sph^2 - K$ et joignant $\infty$ à $\infty$, où l'on identifie les arcs ayant même image et des orientations opposées.
\item Deux sommets sont reliés par une arête si et seulement si ils sont homotopiquement disjoints.
\end{itemize}
\end{definition}

Les graphes $X_\infty$ et $X_r$ sont naturellement munis d'une métrique où toutes les arêtes sont de longueur $1$. Le groupe $\Gamma = MCG(\R^2 - K)$ agit sur $X_\infty$ par isométries.
On adapte ici la preuve de \cite{HPW} de l'hyperbolicité du graphe des arcs dans le cas des surfaces à bord pour montrer l'hyperbolicité de $X_\infty$.

Soient $a$ et $b$ deux arcs simples essentiels sur $\Sph^2 - K$ joignant $\infty$ à $\infty$ et en position d'intersection minimale. On choisit une orientation sur chacun d'entre eux et on note $a^+$, $b^+$ les arcs orientés correspondant. Soit $\pi \in a\cap b$. Soit $a'$, respectivement $b'$, le sous-arc orienté de $a$ commençant comme $a$, respectivement comme $b$, et ayant $\pi$ pour deuxième extrémité. On note $a' \star b'$ la concaténation de ces deux sous-arcs ; en particulier, c'est un arc joignant $\infty$ à $\infty$. On suppose que cet arc est simple. Comme $a$ et $b$ sont en position d'intersection minimale, l'arc $a'\star b'$ est essentiel. Il définit donc un élément de $X_\infty$. On dira que $a' \star b'$ est un \emph{arc unicorne obtenu à partir de $a^+$ et $b^+$}.

On note que cet arc est déterminé de manière unique par le choix de $\pi \in a\cap b$, et que tous les points de $a\cap b$ ne définissent pas un arc sans auto-intersection. Par ailleurs, $a \cap b$ est un ensemble fini, car $a$ et $b$ ont des intersections transverses. Il y a donc un nombre fini d'arcs unicornes obtenus à partir de $a^+$ et $b^+$.

\paragraph{Fait :}\textit{Si $\pi$ et $\pi'$ sont deux points de $a\cap b$ définissant des arcs unicornes $a'\star b'$ et $a'' \star b''$, alors $a'' \subset a'$ si et seulement si $b' \subset b''$.}

\begin{definition}[Ordre total sur les arcs unicornes]
Soient $a^+$ et $b^+$ deux arcs essentiels orientés entre $\infty$ et $\infty$ sur $\Sph^2 -K$, en position minimale d'intersection. On ordonne les arcs unicornes entre $a^+$ et $b^+$ de la manière suivante : \\
$a' \star b' \leq a'' \star b''$ si et seulement si $a'' \subset a'$ et $b' \subset b''$. 
\end{definition}
Cet ordre est total. On note $(c_1,...,c_{n-1})$ l'ensemble ordonné des arcs unicornes antre $a^+$ et $b^+$. Il correspond en particulier à l'ordre des points $\pi$ lorsque l'on parcourt $b^+$.

On définit des chemins unicornes dans le graphe $X_\infty$ de la manière suivante :
\begin{definition}[Chemins unicornes entre arcs orientés] Soient $a^+$ et $b^+$ deux arcs essentiels orientés entre $\infty$ et $\infty$ sur $\Sph^2 -K$, en position minimale d'intersection.
La suite d'arcs unicornes $P(a^+,b^+)=(a=c_0,c_1,...,c_{n-1},c_n=b)$ est appelée  \emph{chemin unicorne entre $a^+$ et $b^+$}.
\end{definition}

\paragraph{Fait :}\textit{Soient $a$ et $b$ deux arcs orientés en position minimale d'intersection et soit $(c_0,...,c_n)$ le chemin unicorne entre ces deux arcs orientés. Soient $a'$ et $b'$ deux arcs en position minimale d'intersection tels que $a'$, respectivement $b'$, est isotope à $a$, respectivement à $b$, et orienté dans le même sens. On note $(d_0,d_1,...,d_{m-1},d_m)$ le chemin unicorne entre $a'$ et $b'$ orientés. Alors $n=m$ et $c_k$ est isotope à $d_k$ pour tout $k$.}
C'est une conséquence de la proposition \ref{prop 3.5}.


\begin{definition}[Chemins unicornes entre éléments de $X_\infty$ orientés] 
 Soient $\alpha^+$ et $\beta^+$ deux éléments de $X_\infty$ munis d'une orientation. Soient $a$ et $b$, deux représentants respectifs de $\alpha$ et $\beta$ qui sont en position minimale d'intersection, munis de l'orientation naturellement induite par $\alpha^+$ et $\beta^+$. Soit $P(a^+,b^+)=(c_0,...,c_n)$ le chemin unicorne associé. Pour tout $1\leq k \leq n$, on note $\gamma_k$ la classe d'isotopie de $c_k$. On pose alors : $P(\alpha^+,\beta^+)=(\gamma_0,\gamma_1,...,\gamma_n)$, qui définit le  \emph{chemin unicorne} entre $\alpha^+$ et $\beta^+$.
\end{definition}

\paragraph{Fait :}\textit{ Tout chemin unicorne est un chemin dans $X_\infty$ ; en effet, pour tout $0\leq i\leq n-1$, $c_i$ et $c_{i+1}$ sont homotopiquement disjoints.} C'est la remarque $3.2$ de \cite{HPW}).

\paragraph{Remarques :} 
\begin{enumerate}
\item Si $a\cap b=\emptyset$, on a alors $P(a^+,b^+)=(a,b)$.
\item Par abus de notation, on notera encore $P(a^+,b^+)$ l'ensemble des éléments de la suite $P(a^+,b^+)$.
\end{enumerate}

Les arcs unicornes ne dépendent que du voisinage de $a \star b$ : si l'on considère un voisinage fermé de $a \star b$ suffisamment petit (pour qu'il soit homotopiquement équivalent à $a \star b$), on peut alors voir les arcs unicornes comme arcs unicornes de la surface compacte à bord définie par ce voisinage. On est alors exactement dans le cas de l'article \cite{HPW}. Cette correspondance nous permet de voir tout chemin unicorne de $X_\infty$ comme un chemin unicorne d'un graphe des arcs d'une surface. En particulier, les lemmes $3.3$, $3.4$, $3.5$ et $4.3$ de \cite{HPW} restent vrais dans $X_\infty$. Comme la proposition $4.2$, puis le théorème $1.2$ en découlent, on obtient de la même façon l'hyperbolicité du graphe $X_\infty$. Il semble difficile de déduire l'hyperbolicité de $X_\infty$ de celle du graphe des arcs d'une seule surface : dans chaque preuve des lemmes, on doit passer par des surfaces différentes, qui dépendent des éléments de $X_\infty$ que l'on considère. Pour plus de commodités, on adapte la preuve de \cite{HPW} dans notre contexte. Le lemme \ref{slim unicorn}, le corollaire \ref{lemme4.3}, le lemme \ref{lemme3.5} et les propositions \ref{prop4.2} et \ref{G_infini hyp} correspondent, dans cet ordre, aux lemmes $3.3$, $4.3$, $3.5$, à la proposition $4.2$ et au théorème $1.2$ de \cite{HPW}.

On note que la preuve de \cite{HPW} ne s'adapte pas directement au graphe des rayons $X_r$ : en effet, l'arc obtenu à partir de deux représentants d'éléments du graphe des rayons orientés de l'infini jusqu'au point d'attachement va de l'infini à l'infini et n'appartient donc pas au graphe des rayons. Si l'on modifie la définition en choisissant l'arc unicorne comme parcourant le début de $a$ puis la fin de $b$, on obtient bien un arc dont la classe d'isotopie est un rayon, mais le lemme \ref{slim unicorn} devient faux, d'où la nécessité de passer par le graphe $X_\infty$.

\subsubsection*{Lemmes sur les chemins unicornes de $X_\infty$}

\begin{lemme}[Les triangles unicornes sont 1-fins] \label{slim unicorn}
Soient $\alpha^+, \beta^+$ et $\delta^+$ trois éléments de $X_\infty$ munis d'une orientation. Alors pour tout $\gamma \in P(\alpha^+,\beta^+)$, l'un des termes $\gamma^*$ de $P(\alpha^+,\delta^+) \cup P(\delta^+,\beta^+)$ est tel que $d(\gamma,\gamma^*)=1$ dans $X_\infty$.
\end{lemme}

\begin{proof}
On choisit des représentants géodésiques $a,b,d$ de $\alpha,\beta,\delta$. Soit $c \in P(a^+,b^+)$ : il existe $a'$ et $b'$ sous-arcs respectifs de $a$ et $b$ tels que $c=a' \star b'$. Si $c$ est disjoint de $d$, $\gamma^*=\delta$ convient. Sinon, soit $d'\subset d$ le sous arc maximal commençant comme $d^+$ et disjoint de $c$. Soit $\sigma \in c$ l'autre extrémité de $d'$. Le point $\sigma$ divise $c$ en deux sous-arcs, dont l'un est contenu dans $a'$ ou $b'$, disons $a'$ (le cas $b'$ est analogue). On note $a''$ ce sous-arc. Alors $c^*:=a'' \star d'$ est un terme de $P(a^+,d^+)$. De plus, $c$ et $c^*$ sont homotopiquement disjoints. \end{proof}

\begin{corollaire}\label{lemme4.3}
Soient $k\in \N$, $m\leq 2^k$ et soit $(\xi_0,...,\xi_m)$ un chemin dans $X_\infty$. On munit les $\xi_i$ d'une orientation arbitraire. Alors $P(\xi_0^+,\xi_m^+)$ est inclus dans un $k$-voisinage de $(\xi_0,...,\xi_m)$.
\end{corollaire}

\begin{proof}
Soit $\gamma \in P(\xi_0^+,\xi_m^+)$. Montrons qu'il existe $i$ tel que $d(\gamma,\xi_i)\leq k$. En appliquant le lemme \ref{slim unicorn} aux sommets $\xi_0^+$, $\xi_m^+$ et $\xi_{E(m/2)}^+$ (où $E(\cdot)$ désigne la partie entière), on obtient $\gamma_1^* \in P(\xi_0^+,\xi_{E(m/2)}^+) \cup P(\xi_{E(m/2)}^+,\xi_m^+)$ tel que $d(\gamma,\gamma_1^*)=1$. On note $(\alpha_1^+,\beta^+_1)$ le couple $(\xi_0^+,\xi_{E(m/2)}^+)$ ou $(\xi_{E(m/2)}^+,\xi_m^+)$ tel que $\gamma_1^* \in P(\alpha_1^+,\beta^+_1)$. \\
On applique alors le lemme \ref{slim unicorn} aux éléments $\alpha_1^+$, $\beta_1^+$ et $\xi_l^+$, où $l$ est choisi de telle sorte que $\xi_l$ est au milieu de $\alpha_1$ et $\beta_1$ sur le chemin $(\xi_0,...,\xi_m)$. On a alors $\gamma_2^* \in P(\alpha_1^+,\xi_l^+) \cup P(\xi_l^+,\beta_1^+)$ tel que $d(\gamma_1^*,\gamma_2^*)=1$, et donc $d(\gamma,\gamma_2^*)\leq 2$. On continue ainsi par récurrence en choisissant à chaque fois un élément $\xi_j$ au milieu des deux éléments concernés par le chemin unicorne précédent, et on finit par trouver $\gamma^*=\xi_i$ tel que $d(\gamma,\gamma_*)\leq k$.
\end{proof}

\begin{lemme} \label{lemme3.5}
Soient $\alpha^+,\beta^+ \in X_\infty$ orientés et soit $P(\alpha^+,\beta^+)=(\gamma_0,...,\gamma_n)$ le chemin unicorne associé dans $X_\infty$. Pour tous $0\leq i \leq j \leq n$, on considère $P(\gamma_i^+,\gamma_j^+)$, où $\gamma_i^+$, respectivement $\gamma_j^+$, a la même orientation que $a^+$, respectivement $b^+$. Alors ou bien $P(\gamma_i^+,\gamma_j^+)$ est un sous-chemin de $P(\alpha^+,\beta^+)$, ou bien $j=i+2$ et $d(\gamma_i,\gamma_j)=1$ dans $X_\infty$.
\end{lemme}

On choisit des représentant $a^+$ et $b^+$, et on note $(c_0,...,c_n)$ le chemin unicorne associé. Pour garder la terminologie de \cite{HPW}, on appelera \emph{demi-bigone} tout bigone ayant l'infini dans son bord. On montre d'abord le sous-lemme suivant :

\begin{sous-lemme}\label{ss-lemme-unicornes}
Soit $c=c_{n-1}$, c'est-à-dire que $c=a'\star b'$, avec l'intérieur de $a'$ disjoint de $b$. Soit $\tilde c$ un arc homotope à $c$ obtenu en poussant $a'$ en dehors de $a$ de telle sorte que $a' \cap \tilde c =\emptyset$. Alors ou bien $\tilde c$ et $a$ sont en position minimale d'intersection, ou bien ces deux arcs bordent exactement un demi-bigone : dans ce cas, après avoir poussé $\tilde c$ à travers ce demi-bigone, obtenant ainsi un arc $\overline{c}$, on a que $\overline c$ et $a$ sont en position minimale d'intersection.
\end{sous-lemme}

\begin{proof}[Preuve du sous-lemme \ref{ss-lemme-unicornes}] Les arcs $\tilde c$ et $a$ ne peuvent pas border un bigone, sinon $a$ et $b$ bordent un bigone, ce qui contredit la position minimale d'intersection. Ainsi si $\tilde c$ et $a$ ne sont pas en position minimale d'intersection, alors ils bordent un demi-bigone $\tilde c' a''$, où $\tilde c' \subset \tilde c$, et $a'' \subset a$ (voir la figure \ref{figu:unicornes} pour un exemple). Soit $\pi'=\tilde c' \cap a''$. Comme $\tilde c$ découpe la sphère en deux composantes connexes, l'une contient $a''$ et l'autre contient $b-b'$, donc l'intérieur de $a''$ est disjoint de $b$. En particulier, $a''$ est situé à la fin de $a$.
De plus, $\pi'$ et $\pi=a'\cap b'$ sont deux points d'intersection de $a\cap b$ successifs sur $b$ (sinon il y a un bigone).

\pgfdeclareimage[interpolate=true,height=5cm]{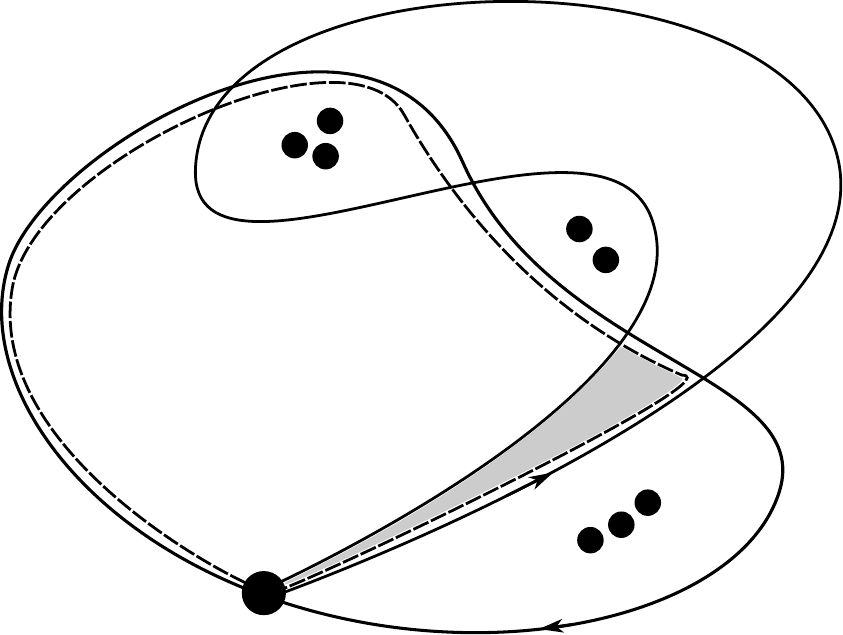}{unicornes}
\begin{figure}[!h]
\labellist
\small\hair 2pt
\pinlabel $b$ at 230 45
\pinlabel $\infty$ at 60 5
\pinlabel $\tilde c$ at 12 100
\pinlabel $a$ at 240 105
\pinlabel $a'$ at 140 30
\pinlabel $a''$ at 137 59
\pinlabel $\pi'$ at 190 90
\pinlabel $\pi$ at 214 73
\pinlabel $K$ at 105 145
\endlabellist
\centering
\includegraphics[scale=0.8]{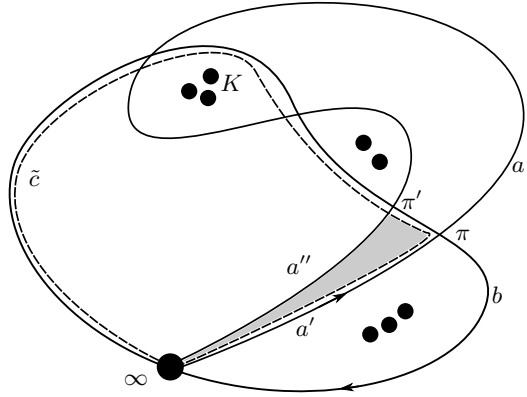}
\vspace{0.3cm}
\caption{Un exemple de deux arcs $a$ et $b$ dans la situation où $a$ et $\tilde c$ (en pointillés) bordent un demi-bigone (grisé). Les points noirs représentent des morceaux de $K$.}
\label{figu:unicornes}
\end{figure}

On note $b''$ la première composante connexe de $b-\pi'$ dans le sens de parcours de $b$. Soit $\overline c := a'' \star b''$. En appliquant à $\overline c$ le même raisonnement que celui appliqué à $\tilde c$, mais en orientant $a$ dans l'autre sens, on obtient que ou bien $\overline c$ et $a$ sont en position minimale d'intersection, ou bien il existe un demi-bigone $\overline c' a'''$, avec $\overline c' \subset \overline c$ et $a''' \subset a$. Mais dans ce dernier cas, on a que $a'''$ est situé sur le début de $a$ (car sur la fin de $a$ orienté à l'envers), d'où $a' \subset a'''$. Comme $\pi'$ est situé avant $\pi$ dans le sens de parcours $b$, on a même $a' \subsetneq a'''$, ce qui contredit le fait que l'intérieur de $a'''$ est disjoint de $b$.
\end{proof}

\begin{proof}[Démonstration du lemme \ref{lemme3.5}]
Si le lemme est vrai pour $i=0$ et $j=n-1$ alors par symétrie il est vrai pour $i=1$ et $j=n$, et donc par récurence il est vrai pour tous $0\leq i \leq j \leq n$. Soit donc $i=0$ et $j=n-1$. On a alors $c_0=a$ et $c_{n-1}=a' \star b'$, où $a'$ intersecte $b$ seulement en son extrémité $\pi$ distincte de l'infini. Soit $\tilde c$ obtenu à partir de $c=c_{n-1}$ comme dans le sous-lemme \ref{ss-lemme-unicornes}. On reprend toutes les notations du sous-lemme \ref{ss-lemme-unicornes}. Si $\tilde c$ est en position minimale d'intersection avec $a$, alors les points de $a \cap b - \{ \pi \}$ qui déterminent des arcs unicornes à partir de $a^+$ et $b^+$ déterminent les mêmes arcs unicornes que ceux réalisés à partir de $a^+$ et $\tilde c^+$, donc le lemme est prouvé dans ce cas.

Sinon, soit $\overline c$ l'arc du sous-lemme \ref{ss-lemme-unicornes}, homotope à $\tilde c$ et en position minimale d'intersection avec $a$ :  les points de $(a \cap b) - \{ \pi , \pi' \}$ qui déterminent des arcs unicornes à partir de $a^+$ et $b^+$ déterminent les mêmes arcs que ceux obtenus à partir de $a^+$ et $\overline c^+$. Soit $a^*=a-a''$. Si $\pi'$ ne détermine pas un arc unicorne à partir de $a^+$ et $b^+$, c'est-à-dire si $a^*$ et $b''$ s'intersectent en dehors de $\pi'$, alors le lemme est montré comme dans le cas précédent.
Sinon, $a^* \star b''=c_1$, puisque c'est le deuxième arc dans la suite des arcs unicornes obtenus à partir de $a^+$ et $b^+$. De plus, comme le sous-arc $\pi \pi'$ de $a$ est dans $a^*$, son intérieur est disjoint de $b''$, donc aussi de $b'$. Ainsi $a^* \star b''$ est juste avant $c$ dans l'ordre des arcs unicornes obtenus à partir de $a^+$ et $b^+$, ce qui signifie que $j=2$, comme voulu.
\end{proof}

\subsubsection*{Hyperbolicité de $X_\infty$}

On peut maintenant déduire des lemmes précédents l'hyperbolicité du graphe considéré.

\begin{prop}\label{prop4.2}
Soit $\mathcal{G}$ un chemin géodésique de $X_\infty$ entre deux sommets $\alpha$ et $\beta$. Alors quelles que soient les orientations choisies sur $\alpha$ et $\beta$, $P(\alpha^+,\beta^+)$ est inclus dans le $6$-voisinage de $\mathcal{G}$.
\end{prop}

\begin{proof}
Soit $\gamma \in P(\alpha^+,\beta^+)$ dont la distance à $\G$ est maximale parmi les éléments de $P(\alpha^+,\beta^+)$. On note $k$ la distance entre $\gamma$ et $\G$ (en particulier, $P(\alpha^+,\beta^+)$ est inclus dans un $k$-voisinage de $\G$). On suppose $k\geq 1$. Si $d(\alpha,\gamma)<2k$, on pose $\alpha':=\alpha$. Sinon, on note $\alpha'$ l'élement le plus proche de $\alpha$ le long de $P(\alpha^+,\beta^+)$ parmi les éléments de $P(\alpha^+,\beta^+)$ à distance $2k$ de $\gamma$. De même, si $d(\beta,\gamma)<2k$, on pose $\beta':=\beta$, et sinon on note $\beta'$ l'élément le proche de $\beta$ le long de $P(\alpha^+,\beta^+)$ parmi les éléments de $P(\alpha^+,\beta^+)$ à distance $2k$ de $\gamma$.
On considère le sous-chemin $\alpha'\beta' \subset P(\alpha^+,\beta^+)$. D'après le lemme \ref{lemme3.5}, $P(\alpha'^+,\beta'^+)$ est un sous-chemin de $P(\alpha^+,\beta^+)$ (on choisit les bonnes orientations sur $\alpha'$ et $\beta'$). Ainsi $\gamma \in P(\alpha'^+,\beta'^+)$ : sinon on est dans le cas $d(\alpha',\beta')=1$, ce qui implique que $\gamma$ est à distance $\leq 1$ de $\alpha$ ou $\beta$, et donc de $\G$.

Soient $\alpha'',\beta'' \in \G$ à distance minimale de $\alpha'$ et $\beta'$ : $d(\alpha'',\alpha')\leq k$ et $d(\beta'',\beta')\leq k$.
Si $\alpha'=\alpha$ ou $\beta'=\beta$, alors $\alpha''=\alpha$ ou $\beta''=\beta$. On a :
$$d(\alpha'',\beta'')\leq d(\alpha'',\alpha')+d(\alpha',\gamma)+d(\gamma,\beta')+d(\beta',\beta'')\leq k+2k+2k+k \leq 6k.$$
Soit $\mathcal{J}$ le chemin de $\alpha'$ à $\beta'$ obtenu en concaténant le sous-chemin $\alpha''\beta''$ de $\mathcal{G}$ avec des chemins géodésiques quelconques entre $\alpha'$ et $\alpha''$, et entre $\beta'$ et $\beta''$. On note $ \xi_0\xi_1...\xi_m$ les sommets de $\mathcal{J}$, et on a $m\leq 8k$. D'après le corollaire \ref{lemme4.3}, il existe $i$ tel que $d(\gamma,\xi_i) \leq E( log_2 8k) $, où $E$ est la fonction partie entière supérieure.

Si $\xi_i \notin \G$, disons $\xi_i \in \alpha \alpha'$, alors on est dans le cas où $d(\gamma,\alpha')=2k$, et donc $d(\gamma,\xi_i) \geq d(\gamma,\alpha') - d(\alpha',\xi_i) \geq k$, d'où $E( log_2 8k) \geq k$. Sinon, si $\xi_i \in \G$, on a directement $E( log_2 8k) \geq k$, cette fois par définition de $k$. Finalement, on a toujours $E( log_2 8k) \geq k$, et donc $k\leq 6$. \end{proof}

\begin{corollaire}\label{coro++}
Soit $\mathcal G$ une géodésique de $X_\infty$ entre deux sommets $\alpha$ et $\beta$. Alors quelles que soient les orientations choisies sur $\alpha$ et $\beta$, $\mathcal{G}$ est incluse dans le $13$-voisinage de $P(\alpha^+,\beta^+)$.
\end{corollaire}

C'est une conséquence de la proposition \ref{copies} et du lemme suivant :
\begin{lemme} Soit $X$ un espace géodésique. Soit $\mathcal G$ une géodésique de $X$ entre deux points $\alpha$ et $\beta$. Soit $k$ un entier positif. Si $\mathcal J$ est un chemin de $X$ entre $\alpha$ et $\beta$ qui reste dans un $k$-voisinage de $\mathcal G$, alors $\mathcal G$ reste dans un $(2k+1)$-voisinage de $\mathcal J$.
\end{lemme}

\begin{proof} [Démonstration du lemme.] Soit $\mathcal G'$ un sous-segment de $\mathcal G$ tel que pour tout $\gamma' \in \mathcal G'$, pour tout $\xi \in \mathcal J$, on a $d(\gamma',\xi)>k$. On montre que tous les points de $\mathcal G'$ sont à distance au plus $(2k+1)$ de $\mathcal J$. On oriente $\mathcal G$ et $\mathcal J$ de $\alpha$ vers $\beta$. L'ensemble $\mathcal G - \mathcal G'$ a deux composantes connexes : on note $\mathcal G_1$ celle située avant $\mathcal G'$ (lorsque l'on parcourt $\mathcal G$ de $\alpha$ vers $\beta$), et $\mathcal G_2$ la deuxième. On a $d(\alpha,\G_2)>k$, sinon $\G'$ est dans le $k$-voisinage de $\beta \in \mathcal J$. Soit $\zeta$ le premier point de $\mathcal J$ (dans le sens de parcours de $\mathcal J$) tel que $d(\zeta,\mathcal G_2)\leq k$.
Soit $\gamma_2 \in \mathcal G_2$ tel que $d(\zeta,\gamma_2)\leq k$. Soit $\zeta' \in \mathcal J$ à distance $1$ de $\zeta$ et situé avant $\zeta$ sur $\mathcal J$. Alors par définition de $\zeta$ et par hypothèse sur $\G'$, il existe $\gamma_1 \in \mathcal G_1$ tel que $d(\zeta',\gamma_1)\leq k$. Ainsi, comme $\mathcal G$ est une géodésique, le segment de $\mathcal G$ entre $\gamma_1$ et $\gamma_2$ est de longueur inférieure ou égal à $2k+1$ et contient $\mathcal G'$. On en déduit que tous les points de $\mathcal G'$ sont à distance au plus $2k+1$ de $\mathcal J$.
\end{proof}

\begin{prop} \label{G_infini hyp}
Le graphe $X_\infty$ est $20$-hyperbolique, au sens de Gromov.
\end{prop}

\begin{proof}
Soit $\alpha \beta \gamma$ un triangle géodésique de $X_\infty$. Soit $\zeta$ sur la géodésique entre $\alpha$ et $\beta$. On choisit des orientations sur $\alpha$, $\beta$ et $\gamma$. D'après le corollaire \ref{coro++}, il existe $\xi$ sur $P(\alpha^+,\beta^+)$ tel que $d(\zeta,\xi) \leq 13$. D'après le lemme \ref{slim unicorn}, il existe $\xi^* \in P(\alpha^+,\gamma^+) \cup P(\gamma^+,\beta^+)$ tel que $d(\xi,\xi^*)\leq 1$. D'après la proposition \ref{prop4.2}, il existe $\zeta^*$ sur un des côtés géodésiques du triangle joignant $\alpha$ à $\gamma$ ou $\gamma$ à $\beta$, tel que $d(\xi^*,\zeta^*)\leq 6$. On a donc $d(\zeta,\zeta^*)\leq 20$, d'où le résultat. \end{proof}

\subsection{Quasi-isométrie entre $X_r$ et $X_\infty$} \label{def-q.i.}

On cherche à déduire l'hyperbolicité du graphe des rayons $X_r$ à partir de celle du graphe $X_\infty$. Pour arriver à cette conclusion, on montre que ces deux graphes sont quasi-isométriques.

\subsubsection*{Rappels de géométrie \og à grande échelle\fg}
On utilise les résultats suivants (voir par exemple \cite{Bridson-Haefliger}).

\begin{definition}[Quasi-isométrie] Soient $X$ et $X'$ deux espaces métriques. Une application $f : X' \rightarrow X$ est un  \emph{plongement $(\kappa,\varepsilon)$-quasi-isométrique} s'il existe $\kappa \geq 1$ et $\varepsilon \geq 0$ tels que pour tous $x,y \in X'$ :
$${1\over \kappa}d(x,y)-\varepsilon \leq d(f(x),f(y))\leq \kappa d(x,y) +\varepsilon.$$

Si de plus il existe $C \geq 0$ tel que tout point de $X$ est dans le $C$-voisinage de $f(X')$, on dit que $f$ est une  \emph{$(\kappa,\varepsilon)$-quasi-isométrie}.

Quand une telle application existe, on dit que $X$ et $X'$ sont  \emph{quasi-isométriques}.
\end{definition}

\paragraph{Remarque :} Le fait d'être quasi-isométriques est une relation d'équivalence.

\begin{definition}[Quasi-géodésique]
Une \emph{$(\kappa,\varepsilon)$-quasi-géodésique} d'un espace métrique $X$ est un plongement $(\kappa,\varepsilon)$-quasi-isométrique d'un intervalle de $\R$ dans $X$. Par abus de langage, on appelle quasi-géodésique tout image dans $X$ d'un tel plongement.
\end{definition}

\begin{theo*}[Lemme de Morse, voir par exemple \cite{Bridson-Haefliger}, th. $1.7$ p.$401$] \label{Morse} Soit $X$ un espace métrique $\delta$-hyperbolique. Pour tous $k, \varepsilon$ réels positifs, il existe une constante universelle $B$ dépendant uniquement de $\delta$, $k$ et $\varepsilon$, telle que tout segment $(\kappa,\varepsilon)$-quasi-géodésique est dans le $B$-voisinage de toute géodésique joignant ses extrémités.
\end{theo*}

 On dira que $B$ est la \emph{constante de Morse} de l'espace $X$.

\begin{theo*}[voir par exemple \cite{Bridson-Haefliger}, th.$1.9$ p.$402$ ]
\label{bridson}
Soient $X$ et $X'$ deux espaces métriques géodésiques et soit $f:X'\rightarrow X$ un plongement quasi-isométrique. Si $X$ est hyperbolique, alors $X'$ est hyperbolique.
\end{theo*}

\subsubsection*{Quasi-isométrie entre $X_r$ et $X_\infty$}

D'après la proposition \ref{G_infini hyp}, on sait que $X_\infty$ est un espace hyperbolique. Pour montrer l'hyperbolicité du graphe des rayons $X_r$, on cherche maintenant à montrer qu'il existe un plongement quasi-isométrique de $X_r$ dans $X_\infty$, ce qui nous permettra de conclure grâce au théorème énoncé ci-dessus. On montre un peu plus, à savoir que le plongement choisi est une quasi-isométrie.

On définit une application $f:X_r \rightarrow X_\infty$ qui à ${x} \in X_r$ associe n'importe quel $\hat {x} \in X_\infty$ tel que ${x}$ et $\hat {x}$ sont homotopiquement disjoints sur $\Sph^2-{(K\cup \{\infty\}})$.

\begin{prop}\label{prop quasi-isom}
L'application $f$ est une quasi-isométrie.
\end{prop}

\begin{lemme} \label{retour}
Soient $\hat {x}, \hat {y} \in X_\infty$ et ${x},{y} \in X_r$ tels que ${x}$ (respectivement ${y}$) est homotopiquement disjoint de $\hat {x}$ (respectivement de $\hat {y}$).
Alors : $$d({x},{y})\leq d(\hat {x},\hat {y})+2.$$
\end{lemme}

\paragraph{Remarque :} En particulier, on note que ce lemme implique que pour tous $x,y \in X_r$, $d(x,y)-2\leq d(f(x),f(y))$.

\begin{proof}
On note $n:=d(\hat {x},\hat {y})$. Soit $(\hat \mu_j)_{0\leq j\leq n}$ une géodésique dans $X_\infty$ entre $\hat {x}$ et $\hat {y}$ (en particulier, $\hat \mu_0 = \hat {x}$ et $\hat \mu_n=\hat {y}$). On va construire un chemin $(\mu_1,...,\mu_{n-1})$ de longueur $(n-1)$ dans $X_r$, puis montrer que $d(x,\mu_1)\leq 2$ et $d(\mu_{n-1},y)\leq 2$. Pour tout élément $\alpha$ de $X_r$ ou $X_\infty$, on note toujours $\alpha_\#$ le représentant géodésique de $\alpha$.

Comme $(\hat \mu_i)_i$ est une géodésique de $X_\infty$, pour tout $1\leq i \leq n-1$, $(\hat \mu_i)_\#$ est disjointe de $(\hat \mu_{i-1})_\#$ et $(\hat \mu_{i+1})_\#$ (sauf en $\{\infty\}$), et $(\hat \mu_{i-1})_\#$ et $(\hat \mu_{i+1})_\#$ s'intersectent ailleurs qu'en l'infini. Ainsi $(\hat \mu_i)_\#$ sépare la sphère $\Sph^2$ en deux composantes connexes, dont l'une contient $(\hat \mu_{i-1})_\#$ et $(\hat \mu_{i+1})_\#$. On note $A_i$ l'autre composante connexe. On remarque que pour tout $1\leq i \leq n-2$, $A_i$ est disjointe de $A_{i+1}$. Pour tout $1\leq i \leq n-1$, on choisit un rayon $\mu_i$ tel que $(\mu_i)_\#$ est inclus dans $A_i$ (un tel $\mu_i$ existe car les $\hat \mu_i$ sont des courbes essentielles). On a donc construit un chemin $(\mu_i)_{1\leq i\leq n-1}$ de longueur $(n-1)$ dans $X_r$.

Montrons maintenant que $d(x,\mu_1)\leq 2$ : si $(\mu_1)_\#$ intersecte $x_\#$, alors $x_\#$ est dans la composante connexe de $\Sph^2 - \hat x_\#$ qui contient $(\hat \mu_1)_\#$ et $(\mu_1)_\#$. Tout représentant de rayon inclus dans l'autre composante connexe de $\Sph^2 - \hat x_\#$ n'intersecte ni $(\mu_1)_\#$, ni $x_\#$ : ainsi $d(x,\mu_1)\leq 2$.
On montre de même que $d(\mu_{n-1},y)\leq 2$.\end{proof}

\begin{lemme}\label{lemme X_r X_infty}
Soit $\hat x \in X_\infty$. Soit $x \in X_r$ homotopiquement disjoint de $\hat x$. Alors : $$d(f(x),\hat x) \leq 2.$$
\end{lemme}

\begin{proof}
On note toujours $x_\#$ et $\hat x_\#$ les représentants géodésiques de $x$ et $\hat x$, qui sont disjoints (sauf en l'infini). Comme $x_\#$ est disjoint de $\hat x_\#$, il existe un disque topologique ouvert $\mathcal{U}$ de $\Sph^2$ contenant $x_\#-\{ \infty\}$ et disjoint de $\hat x_\# -\{ \infty\}$. De même, comme $f(x)_\#$ est disjoint de $x_\#$, on a un disque topologique ouvert $\mathcal{V}$ contenant $x_\#-\{ \infty\}$ et disjoint de $f(x)_\# -\{ \infty\}$. Ainsi $\mathcal{U} \cap \mathcal{V}$ contient un disque topologique ouvert contenant $x_\#-\{ \infty\}$ et disjoint de $(\hat x_\# \cup f(x)_\#) -\{ \infty\}$. En particulier, $\mathcal{U} \cap \mathcal{V}$ contient des points de $K$, puisqu'il contient le point d'attachement de $x_\#$. Il existe donc $\hat y_0 \subset (\mathcal{U} \cap \mathcal{V}) - K$ une courbe simple de $\Sph_2$ passant par $\infty$, ayant pour classe d'isotopie l'élément $\hat y \in X_\infty$. Finalement, $d(\hat y,\hat x)=d(\hat y ,f(x))=1$, d'où le résultat.
\end{proof}

\begin{lemme} \label{quasi-iso}
Pour tous $x,y \in X_r$, on a : $$ d(f(x),f(y))\leq d(x,y)+4.$$
\end{lemme}

\begin{proof}
Soient $x,y \in X_r$ et $n=d(x,y)$.
Si $x$ et $y$ n'ont pas le même point d'attachement, on choisit une géodésique $(\gamma_i)_{0\leq i \leq n}$ de $X_r$ entre $x$ et $y$, telle que pour tout $i,j$, les éléments $\gamma_i$ et $\gamma_j$ n'ont pas le même point d'attachement. Des tels $\gamma_i$ existent quitte à changer certains points d'attachement pour un point voisin de $K$ sans ajouter de point d'intersection avec les autres $\gamma_k$. Si $x$ et $y$ ont le même point d'attachement, on choisit pour $\gamma_n$ un rayon homotopiquement disjoint de $y$ et de $f(y)$ et qui n'a pas le même point d'attachement que $x$, puis on choisit une géodésique $(\gamma_i)_{0\leq i \leq n}$ de $X_r$ entre $x=\gamma_0$ et $\gamma_n$.

Autour des points d'attachement des rayons $\gamma_i$, on choisit maintenant des petits voisinages deux à deux disjoints et tels que chaque voisinage intersecte un unique rayon, qui vient s'attacher à un point contenu dans le voisinage. Si $y \neq \gamma_n$, on choisit de plus $\hat \gamma_n$ disjoint de $y$. On définit alors pour chacun des rayons $\gamma_i$ une courbe $\hat \gamma_i$ de la manière suivante : on parcours $\gamma_i$ jusqu'au voisinage choisi, on parcourt le bord du voisinage choisi puis on reparcourt $\gamma_i$ dans l'autre sens. On obtient ainsi un élément de $X_\infty$.

Par construction, pour tout $i$ entre $2$ et $n-1$, on a $d(\hat \gamma_{i-1},\hat \gamma_i)=d(\gamma_i,\gamma_{i+1})=1$.
D'après le lemme \ref{lemme X_r X_infty} appliqué à $\hat \gamma_0$ disjoint de $x=\gamma_0$ et à $\hat \gamma_n$ disjoint de $y$, on obtient $d(\hat \gamma_0,f(x))\leq 2$ et $d(\hat \gamma_n,f(y))\leq 2$.

Finalement on a $d(f(x),f(y))\leq n+4$.
\end{proof}

\paragraph{Fin de la preuve de la proposition \ref{prop quasi-isom} :}
Les lemmes \ref{retour} (pour la première inégalité) et \ref{quasi-iso} (pour la deuxième inégalité) nous donnent : $d(x,y) -2 \leq d(f(x),f(y)) \leq d(x,y)+4$.\\
Le lemme \ref{lemme X_r X_infty} nous donne la constante $C=2$ telle que tout $\hat x$ de $X_\infty$ soit dans un $C$-voisinage de $f(X_r)$, ce qui termine la preuve de la proposition \ref{prop quasi-isom}. \qed

\subsubsection*{Hyperbolicité du graphe des rayons}

Finalement on a montré le théorème suivant :
\begin{theo}\label{theo-hyperbolique}
Le graphe des rayons est hyperbolique au sens de Gromov.
\end{theo}

\begin{proof}
C'est une conséquence de la proposition \ref{prop quasi-isom} (il existe un plongement quasi-isométrique de $X_r$ dans $X_\infty$), de la proposition \ref{G_infini hyp} ($X_\infty$ est hyperbolique) et du théorème \ref{bridson} (si $f:X\rightarrow X'$ est un plongement quasi-isométrique et $X'$ est hyperbolique, alors $X$ est hyperbolique).
\end{proof}

\section{Quasi-morphismes non triviaux}\label{section-qm}

Dans \cite{Bestvina-Fujiwara}, Mladen Bestvina et Koji Fujiwara montrent que l'espace des quasi-morphismes non triviaux sur le groupe modulaire d'une surface compacte est de dimension infinie. Ils montrent d'abord (théorème $1$ de \cite{Bestvina-Fujiwara}) que si un groupe $G$ agit par isométries sur un espace hyperbolique $X$, alors sous la condition d'existence d'éléments hyperboliques qui vérifient certaines propriétés dans $G$, l'espace des quasi-morphismes non triviaux sur ce groupe est de dimension infinie. Dans un deuxième temps, ils montrent que si l'action de $G$ sur $X$ est \og faiblement proprement discontinue\fg, alors il existe des éléments qui vérifient les conditions du théorème $1$, puis ils montrent que les groupes modulaires de surfaces compactes agissent proprement faiblement discontinûment sur les complexes de courbes associés. \\

Dans le cas du groupe $\Gamma$ qui nous intéresse, l'action considérée sur l'espace $X_r$ n'est pas faiblement proprement discontinue. On dit qu'un élément $g$ d'un groupe $G$ agit proprement faiblement discontinûment sur un espace hyperbolique $X$ si pour tout $x \in X$, pour tout $C >0$, il existe $N>0$ tel que le nombre de $\sigma \in G$ vérifiant $d(x,\sigma x) \leq C$ et $d(g^N x,\sigma g^N x)\leq C $ est fini (voir par exemple \cite{SCL} p74).

\paragraph{Fait :}  \label{rq-WPD} \textit{Pour tout $g\in \Gamma$, l'action de $g$ sur le graphe des rayons $X_r$ n'est pas proprement faiblement discontinue.}

 En effet, pour tout $x \in X_r$, pour tout $N\in \N$, il existe une infinité de $\sigma \in MCG(\R^2-Cantor)$ tels que $d(x,\sigma x) \leq 1$ et $d(g^N x,\sigma g^N x)\leq 1$ : on considère un voisinage $\mathcal{U}$ d'un point du Cantor tel que $\mathcal{U}$  est disjoint de $x$ et de $g^N x$, alors tout $\sigma$ à support dans $\mathcal{U}$  fixe $x$ et $g^N x$ donc vérifie $d(x,\sigma x) \leq 1$ et $d(g^N x,\sigma g^N x)\leq 1$. De plus il y a une infinité de tels $\sigma$ car il y a une infinité de points de l'ensemble de Cantor dans $\mathcal{U}$.

  La stratégie de \cite{Bestvina-Fujiwara} ne s'applique donc pas entièrement, mais on peut trouver explicitement des éléments de $\Gamma$ qui vérifient les hypothèses du théorème $1$ de \cite{Bestvina-Fujiwara}, ce qui nous permet de montrer que l'espace des quasi-morphismes non triviaux sur $\Gamma$ est de dimension infinie.\\

On commence par trouver un élément $h\in \Gamma$ qui agit par translation sur l'axe des $(\alpha_k)_{k}$ défini précédemment. On montre ensuite, en utilisant le \og nombre d'intersections positives\fg, que si $w$ est un sous-segment suffisamment long de cet axe, alors pour tout $g\in \Gamma$, $g$ ne peut pas retourner ce segment dans un voisinage proche de l'axe (proposition \ref{copies}). On utilisera enfin cette proposition d'une part pour construire un quasi-morphisme non trivial explicite sur $\Gamma$ et d'autre part pour construire des éléments de $\Gamma$ qui vérifient les conditions du théorème $1$ de \cite{Bestvina-Fujiwara}.

\subsection{Un élément de $\Gamma$ qui agit par translation sur un axe géodésique infini du graphe des rayons}

On cherche à définir un élément hyperbolique $h \in \Gamma$ comme sur la figure \ref{figu:h}, où chaque brin envoie un sous-ensemble de Cantor sur un autre, de telle sorte que cet élément envoie $\alpha_k$ sur $\alpha_{k+1}$ pour tout $k \in \N$ (voir figure \ref{figu:action-h}).

\begin{figure}[!h]
\labellist
\small\hair 2pt
\pinlabel $K_{-4}$ at 4 204
\pinlabel $K_{-3}$ at 56 204
\pinlabel $K_{-2}$ at 106 204
\pinlabel $K_{-1}$ at 156 204
\pinlabel $K_0$ at 264 204
\pinlabel $K_1$ at 312 204
\pinlabel $K_2$ at 362 204
\pinlabel $K_3$ at 413 204
\pinlabel $\infty$ at 211 204
\pinlabel $h$ at 465 110
\endlabellist
\centering
\vspace{0.4cm}
\includegraphics[scale=0.5]{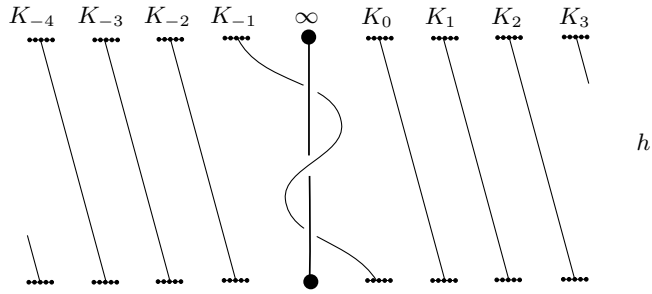}
\caption{Représentation de l'élément $h \in \Gamma$.}
\label{figu:h}
\end{figure}

Comme $(\alpha_k)_{k\in \N}$ est un demi-axe géodésique (d'après la proposition \ref{alpha_k geod}), on a alors que $(h^n(\alpha_0))_{n\in \Z}$ est un axe géodésique du graphe des rayons, sur lequel l'élément $h$ agit par translation.

\begin{figure}[!h]
\labellist
\small\hair 2pt
\pinlabel $K_{-1}$ at 155 317
\pinlabel $K_0$ at 264 317
\pinlabel $K_1$ at 314 317
\pinlabel $\infty$ at 214 317
\pinlabel $h$ at 451 256
\pinlabel $h$ at 451 121
\pinlabel $\alpha_0$ at 239 281
\pinlabel $\alpha_1$ at 309 154
\pinlabel $\alpha_2$ at 351 12
\endlabellist
\centering
\vspace{0.7cm}
\includegraphics[scale=0.5]{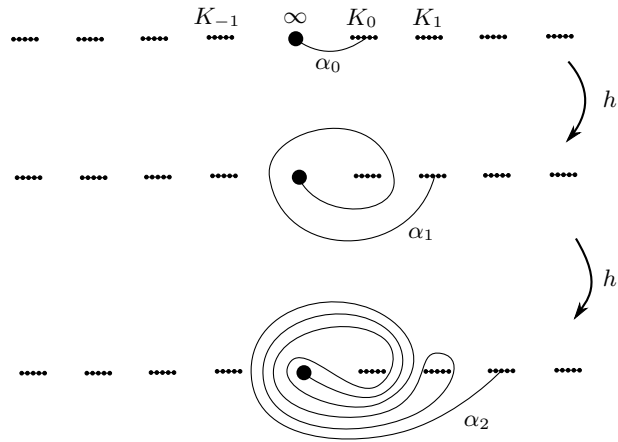}
\caption{Action de $h$ sur les rayons $\alpha_0$ et $\alpha_1$.}
\label{figu:action-h}
\end{figure}

\subsubsection*{Définition de $h$}

On fixe un équateur $\E$ et un alphabet de segments $(s_k)_{k\in \Z}$ comme dans la partie \ref{partie codage}. Pour tout $k \in \Z- \{0\}$, on note $C_k$ les points de $K$ entre $s_{k-1}$ et $s_k$ sur $\E$. En particulier, les $C_k$ sont des ouverts-fermés de l'ensemble de Cantor initial $K$ pour tout $k$, ce sont donc des ensembles de Cantor (tout ouvert-fermé d'un ensemble de Cantor en est un, d'après la caractérisation comme compact métrique totalement discontinu sans point isolé). On note $I$ une composante connexe de $\E - K$ telle que $I \cup \{\infty\}$ partagent l'équateur en deux composantes dont l'une contient tous les segments $s_k$ avec $k>0$ et l'autre contient tous les segments $s_k$ avec $k<0$.  

Soit $\C_N$ un cercle topologique qui coïncide avec l'équateur $\E$ en dehors d'un voisinage de l'infini et qui passe dans l'hémisphère nord au dessus de l'infini. Soit $\C_S$ un cercle topologique qui coïncide avec l'équateur $\E$ en dehors d'un voisinage de l'infini et qui passe dans l'hémisphère sud en dessous de l'infini. 

\pgfdeclareimage[interpolate=true,height=5cm]{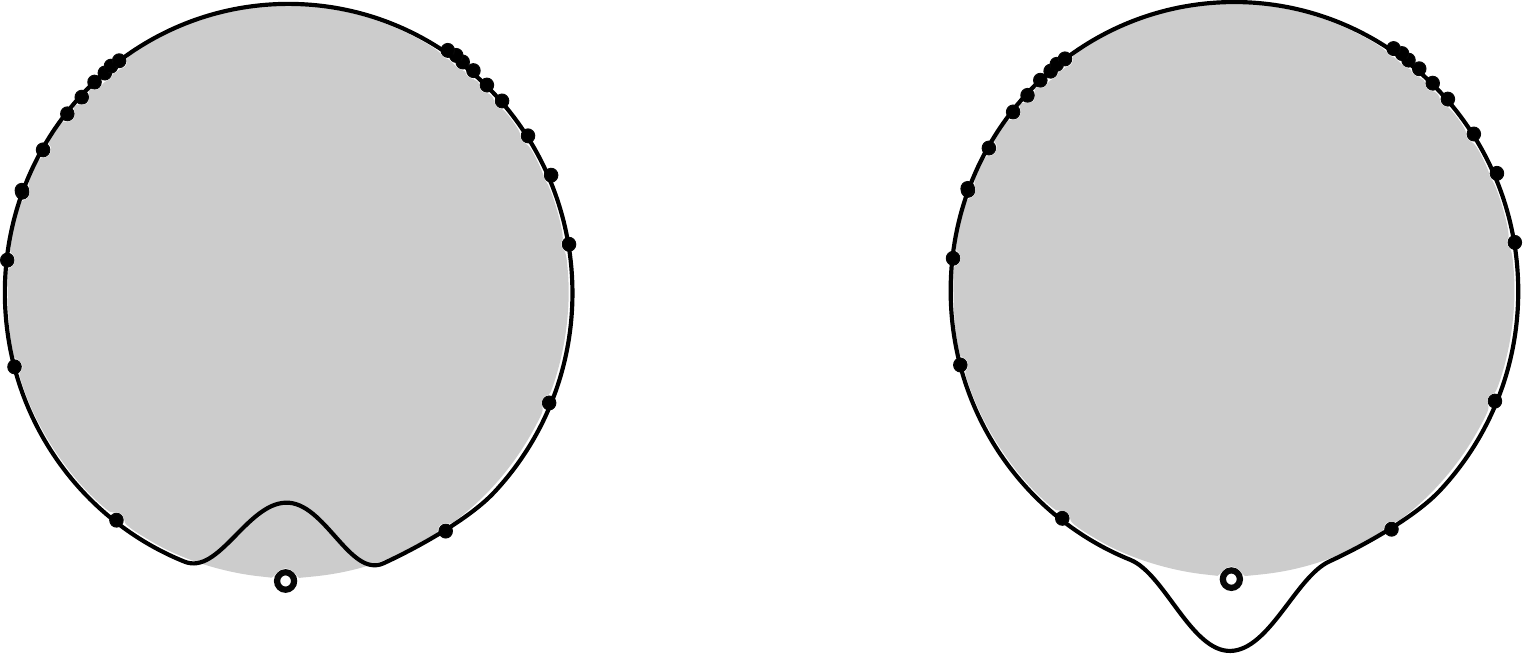}{cercles}
\begin{figure}[!h]
\labellist
\small\hair 2pt
\pinlabel $\infty$ at 82 10
\pinlabel $\infty$ at 355 10
\pinlabel $I$ at 79 200
\pinlabel $I$ at 356 200
\pinlabel $\C_N$ at 25 190
\pinlabel $\C_S$ at 300 190
\pinlabel $K_0$ at 133 20
\pinlabel $K_0$ at 406 21
\pinlabel $K_{-1}$ at 25 27
\pinlabel $K_{-1}$ at 299 28
\pinlabel $K_1$ at 175 67
\pinlabel $K_1$ at 447 68
\endlabellist
\centering
\vspace{0.3cm}
\includegraphics[scale=0.6]{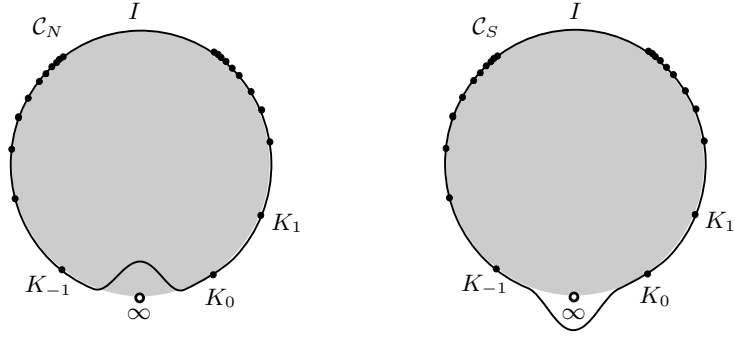}
\caption{$\C_N$ et $\C_S$ (sur chaque figure, la partie grisée représente l'hémisphère nord).}
\label{cercles}
\end{figure}

Soit $\tilde t_1$ un homéomorphisme de $\C_N$ qui envoie chaque morceau d'ensemble de Cantor $K_k$ sur le morceau d'ensemble de Cantor $K_{k+1}$ pour tout $k\in \Z$ et qui vaut l'identité sur $I$ (tout ensemble de Cantor de l'intervalle peut être envoyé sur tout autre ensemble de Cantor par un homéomorphisme de l'intervalle). On prolonge $\tilde t_1$ à un homéomorphisme de la sphère $\Sph^2$, et on considère sa classe d'isotopie $t_1 \in \Gamma$ (voir figure \ref{figu:3h}).

\vspace{0.3cm}
\begin{figure}[!h]
\labellist
\small\hair 2pt
\pinlabel $t_1$ at 456 145
\pinlabel $t_2$ at 456 88
\pinlabel $t_1$ at 456 31
\endlabellist
\centering
\includegraphics[scale=0.5]{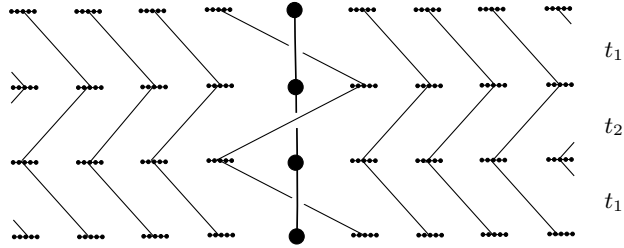}
\caption{Définition de $h:=t_1t_2t_1$.}
\label{figu:3h}
\end{figure}

De même, soit $\tilde t_2$  un homéomorphisme de $\C_S$ qui envoie chaque morceau d'ensemble de Cantor $K_{k+1}$ sur le morceau d'ensemble de Cantor $K_{k}$ pour tout $k\in \Z$ et qui fixe $I$. On prolonge $\tilde t_2$ à un homéomorphisme de la sphère $\Sph^2$, et on considère sa classe d'isotopie $t_2 \in \Gamma$. En particulier, si on note $\phi$ la classe d'isotopie de la rotation d'angle $\pi$ autour de $\infty$ qui envoie pour tout $k\in \Z$ le morceau d'ensemble de Cantor $K_k$ sur le morceau $K_{-k-1}$, alors on peut choisir $t_2 = \phi t_1 \phi^{-1}$.\\

On pose enfin $h:=t_1 t_2 t_1$.

\subsubsection*{Action de $h$ sur le graphe des rayons}

S'il existe une géodésique qui est globalement invariante par une isométrie $g$, et si $g$ n'a pas de point fixe sur cette géodésique, alors on dit que $g$ est hyperbolique et que cette géodésique est un axe de $g$.

\begin{theo}\label{theo-halphak}
L'action de $h$ sur le graphe des rayons est hyperbolique, d'axe $(\alpha_k)_k$. Plus précisément, $h(\alpha_k)=\alpha_{k+1}$ pour tout $k\in \N$.
\end{theo}

Pour voir que $h(\alpha_k)=\alpha_{k+1}$ pour tout $k\geq 0$, on représente $\alpha_k$ par un graphe, comme sur la figure \ref{diag-alpha2}. 

\vspace{0.3cm}
\begin{figure}[h]
\labellist
\small\hair 2pt
\pinlabel $3$ at 365 -5
\pinlabel $1$ at 405 0
\pinlabel $1$ at 415 60
\endlabellist
\centering
\includegraphics[scale=0.6]{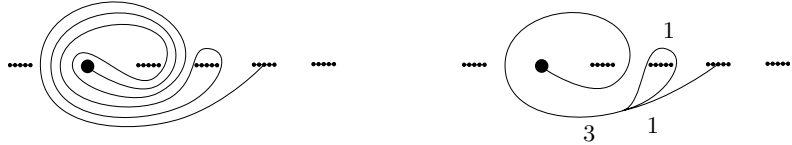}
\caption{A gauche, le rayon $\alpha_2$ ; à droite, un graphe le représentant.}
\label{diag-alpha2}
\end{figure}

Pour chaque rayon, on peut choisir une courbe le représentant et identifier certains morceaux de courbes qui restent dans un voisinage tubulaire les uns des autres. On obtient ainsi un graphe fini plongé de façon lisse dans $\Sph^2$ et disjoint de tous les points de $K$ sauf du point d'attachement du rayon initial. En chaque noeud, les arêtes se répartissent en deux directions. Chaque arête porte un poids qui correspond au nombre de morceaux de courbes qu'elle représente : en chaque noeud, dans une des deux directions il y a une seule arête, qui porte un poids égal à la somme des poids des arêtes de l'autre direction. On peut retrouver le rayon initial à partir d'un graphe le représentant : en effet, il suffit de dupliquer chaque arête le nombre de fois correspondant à son poids, et de recoller les morceaux ainsi obtenus en chaque noeud de l'unique façon possible. On ne recolle que des morceaux arrivant sur le noeud avec des directions différentes, et on cherche à obtenir une courbe simple donc il y a un ordre bien défini sur les morceaux de courbes pour faire ce recollement.

\vspace{0.3cm}
\begin{figure}[!h]
\labellist
\small\hair 2pt
\pinlabel $K_k$ at 270 78
\pinlabel $1$ at 235 85
\pinlabel $3^{k-3}$ at 170 86
\pinlabel $3^{k-2}$ at 127 86
\pinlabel $1$ at 240 5
\pinlabel $3^{k-3}$ at 160 5
\pinlabel $3^{k-2}$ at 110 5
\pinlabel $3^{k-1}$ at 60 5
\endlabellist
\centering
\includegraphics[scale=0.6]{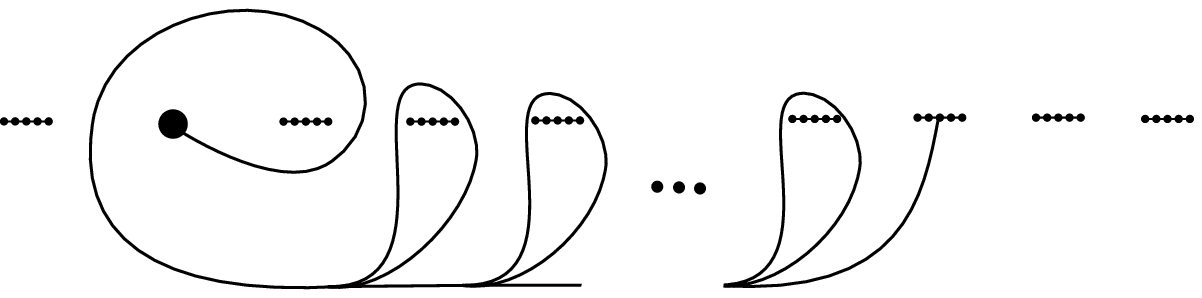}
\caption{Exemple d'un graphe représentant le rayon $\alpha_k$.}
\label{diag-alphak}
\end{figure}

\vspace{0.3cm}
\begin{figure}[!h]
\labellist
\small\hair 2pt
\pinlabel $t_1$ at 455 404
\pinlabel $t_2$ at 454 267
\pinlabel $t_1$ at 454 134

\pinlabel $\alpha_k$ at 424 456
\pinlabel $t_1(\alpha_k)$ at 429 326
\pinlabel $t_2t_1(\alpha_k)$ at 429 196
\pinlabel $h(\alpha_k)$ at 429 56

\pinlabel $1$ at 295 405
\pinlabel $3^{k-3}$ at 210 405
\pinlabel $3^{k-2}$ at 165 405
\pinlabel $3^{k-1}$ at 115 405
\pinlabel $1$ at 285 480
\pinlabel $3^{k-3}$ at 220 483
\pinlabel $3^{k-2}$ at 182 483

\pinlabel $1$ at 320 13
\pinlabel $3^{k-3}$ at 240 13
\pinlabel $3^{k-2}$ at 200 13
\pinlabel $3^{k-1}$ at 140 13
\pinlabel $1$ at 317 90
\pinlabel $3^{k-3}$ at 257 92
\pinlabel $3^{k-2}$ at 218 92

\pinlabel $K_k$ at 320 473
\pinlabel $K_{k+1}$ at 354 81
\endlabellist
\centering
\includegraphics[scale=0.6]{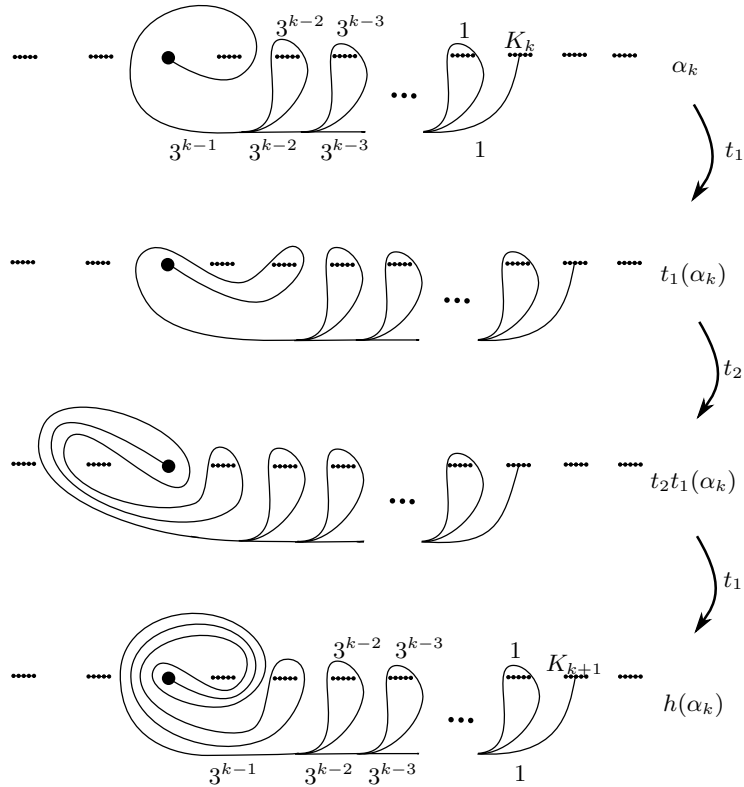}
\caption{Image de $\alpha_k$ par $h$.}
\label{diag-h-alphak}
\end{figure}

Sur la figure \ref{diag-alphak}, on a dessiné un graphe particulier représentant $\alpha_k$, pour tout $k \geq 0$. Comme il existe une courbe $a_k$ représentant $\alpha_k$ qui reste dans un voisinage tubulaire de ce graphe, si $h_0$ est un représentant de $h$, on a que $h_0(a_k)$ reste dans un voisinage tubulaire de l'image par $h_0$ du graphe : le rayon correspondant à l'image du graphe est $h(\alpha_k)$. Sur la figure \ref{diag-h-alphak}, on a dessiné un graphe représentant $\alpha_k$ et les images successives de ce graphe par des représentants de $t_1$, $t_2$ et $t_1$. Le graphe final est donc l'image du graphe de $\alpha_k$ par $h$ : il représente $h(\alpha_k)$. Par ailleurs, on voit que le rayon représenté par ce graphe final est $\alpha_{k+1}$ : ainsi $h(\alpha_k)=\alpha_{k+1}$ pour tout $k \in \N$.

\subsection{Nombre d'intersections positives}

\subsubsection*{Définition}

On note $X_r$ le graphe des rayons, et on oriente chaque rayon de l'infini jusqu'à son point d'attachement.

\begin{definition}[Nombre d'intersections positives] Soit $I : X_r^2 \rightarrow \N \cup \{\infty\}$ l'application qui à deux rayons $\alpha$ et $\beta$ de $X_r$ associe le nombre d'intersections positives entre deux représentants de $\alpha$ et $\beta$ en position minimale d'intersection (voir figure \ref{intersections}).
\end{definition}

\pgfdeclareimage[interpolate=true,height=5cm]{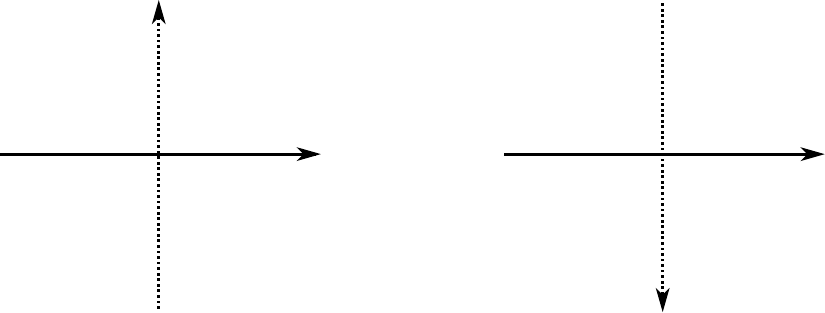}{intersections}
\begin{figure}[!h]
\labellist
\small\hair 2pt
\pinlabel $\alpha$ at 150 50
\pinlabel $\beta$ at 200 3
\pinlabel $\alpha$ at 3 50
\pinlabel $\beta$ at 53 3
\endlabellist
\centering
\includegraphics[scale=0.8]{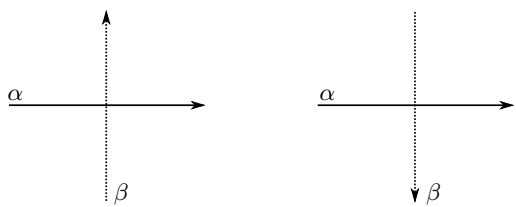}
\vspace{0.2cm}
\caption{Intersection positive à gauche, négative à droite}
\label{intersections}
\end{figure}

\paragraph{Remarques :}
\begin{enumerate}
\item Ce nombre est bien défini : il ne dépend pas du choix de représentants en position minimale d'intersection (d'après la proposition \ref{prop 3.5}).
\item En général, on a $I(\alpha,\beta)\neq I(\beta,\alpha)$.
\item Pour tout $g \in \Gamma$, $I(g.\alpha,g.\beta)=I(\alpha,\beta)$ (car $\Gamma$ est obtenu comme quotient du groupe des homéomorphismes préservant l'orientation). \label{rq3}
\end{enumerate}

\subsubsection*{Cas de la suite $(\alpha_k)_k$}

\begin{lemme} \label{lemme-I1}
Soient $\beta$ et $\gamma$ deux éléments de $X_r$ tels que $A(\gamma)\leq A(\beta)-2$, où $A$ est l'application définie dans la partie \ref{def de A}. Alors $I(\gamma,\beta)\geq 1$.
\end{lemme}

\begin{proof}
On pose $n:=A(\beta)$. Alors $\gamma$ ne commence pas par $\mathring \alpha_{n-1}$. Sur la figure \ref{figu:lemmeI1}, on a représenté le début de $\beta$, c'est-à-dire $\mathring \alpha_n$. Tout représentant de $\gamma$ part de l'infini et doit s'attacher à un point de l'ensemble de Cantor : ainsi, tout représentant de $\gamma$ doit sortir de la zone grisée. Comme $\gamma$ ne commence pas par $\mathring \alpha_{n-1}$, $\gamma$ ne peut pas sortir de la zone grisée en coupant $s_{-1}$. Ainsi $\gamma$ sort de cette zone en intersectant $\beta$. La première intersection est positive, et on a donc $I(\gamma,\beta)\geq 1$.
\end{proof}

\pgfdeclareimage[interpolate=true,height=5cm]{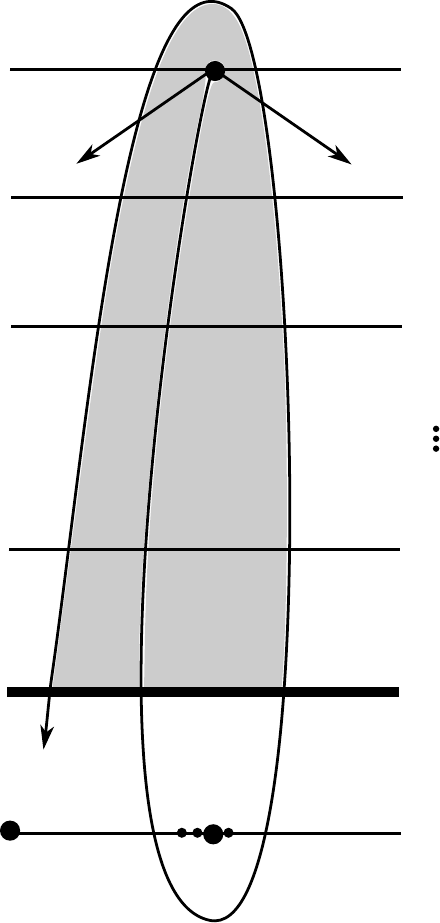}{lemmeI1}
\begin{figure}[!h]
\labellist
\small\hair 2pt
\pinlabel $\mathring \alpha_n$ at 5 50
\pinlabel $s_{-1}$ at 127 65
\pinlabel $s_{-1}$ at -10 245
\pinlabel $s_0$ at 127 245
\pinlabel $s_1$ at 127 210
\pinlabel $\mathring \gamma$ at 100 229
\pinlabel $\mathring \gamma$ at 20 229
\pinlabel ${p}_{n-1}$ at 70 15
\endlabellist
\centering
\includegraphics[scale=0.8]{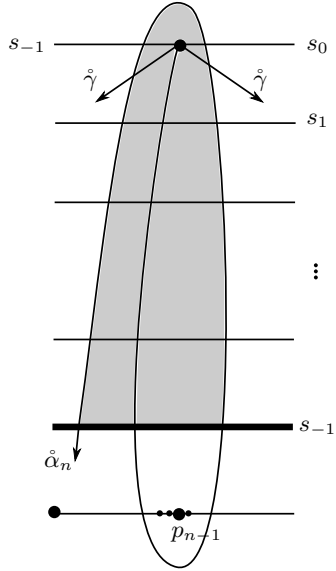}
\caption{Illustration du lemme \ref{lemme-I1} (par définition de $(\alpha_k)_k$, il n'y a aucun point de $K$ dans la zone grisée).}
\label{figu:lemmeI1}
\end{figure}

\paragraph{Remarques :} \begin{itemize}
\item Comme $\alpha_0$ et $\alpha_1$ sont homotopiquement disjoints, on a : \\
$I(\alpha_0,\alpha_1)=I(\alpha_1,\alpha_0)=0$.
\item On n'utilisera pas ce résultat dans la suite, mais on peut calculer précisément les nombres d'intersections positives entre $\alpha_0$ et $\alpha_k$ pour tout $k \geq 2$. On a :
$$I(\alpha_0,\alpha_k)={{3^{k-1}+2k-3}\over 4} \texttt{\ et\ } I(\alpha_k,\alpha_0) = {{3^{k-1}-2k+1}\over 4}.$$
 En effet, notons $(p_k,n_k)=(I(\alpha_0,\alpha_k),I(\alpha_k,\alpha_0))$. On a alors : $$(p_{k+1},n_{k+1})=(2p_k+n_k+1,p_k+2n_k).$$ Ceci vient de la construction de $(\alpha_k)_k$ : on trace un tube autour de $\alpha_{k-1}$, et on peut donc regarder l'orientation des intersection entre ce tube et $\alpha_0$. On sait alors exprimer $p_k$ et $n_k$ en fonction de $k$.
\end{itemize}

\subsection{Non-retournement de l'axe $(\alpha_k)_k$}

On note $B$ la $(2,4)$-constante de Morse du graphe des rayons (voir la partie \ref{def-q.i.}). Quitte à prendre sa partie entière supérieure, on suppose que $B$ est un entier. On cherche à montrer une proposition de non-retournement de l'axe $(\alpha_k)_k$ (proposition \ref{copies}), qui nous servira de manière fondamentale dans les constructions de quasi-morphismes non triviaux (proposition \ref{prop-qm} et théorème \ref{dim infinie}). Pour montrer cette proposition, on a besoin de pouvoir comparer les orientations de certains segments.

\subsubsection*{Segments orientés dans le même sens}

Soit $X$ un espace métrique géodésique. Soient $\gamma_1=[p_1q_1]$ et $\gamma_2=[p_2q_2]$ deux segments géodésiques de $X$ de même longueur et orientés de $p_i$ vers $q_i$. Soit $\gamma'_1$ un segment géodésique (éventuellement infini) contenant $\gamma_1$ et soit $C$ une constante telle que $\gamma_2$ est inclus dans un $C$-voisinage de $\gamma'_1$ et telle que $d(p_1,p_2)\leq C$. On suppose de plus que $|\gamma_1|=|\gamma_2|\geq 3C$. 
Dans ces conditions, on dira que $\gamma_1$ et $\gamma_2$ sont \emph{orientés dans le même sens} si pour tout $r\in \gamma'_1$ tel que $d(r,q_2)\leq C$, $r$ est du même côté de $p_1$ que $q_1$ sur $\gamma'_1$. Comme $\gamma_1$ et $\gamma_2$ sont de longueur supérieure à $3C$, on vérifie facilement que l'existence d'un seul $r$ vérifiant ces conditions suffit.

\begin{lemme}\label{géod-dist}
Si $\gamma_1$ et $\gamma_2$ sont les segments décrits précédemment et s'ils sont orientés dans le même sens, alors $d(q_1,q_2)\leq 3C$.
\end{lemme}

\begin{proof}
Soit $r$ sur $\gamma'_1$ tel que $d(q_2,r)\leq C$. On note $\alpha$ le segment de $\gamma'_1$ entre $p_1$ et $r$, $\beta$ celui entre $r$ et $q_1$. 

\paragraph{1er cas : si $r \in \gamma_1$.} On a alors : $$|\gamma_2|=|\gamma_1|=|\alpha|+|\beta|\leq d(p_1,p_2)+|\alpha|+C.$$
On en déduit : $$|\beta| \leq d(p_1,p_2)+C \leq 2C.$$
Finalement on obtient : $$d(q_1,q_2) \leq d(q_1,r)+d(r,q_2) \leq |\beta| +C \leq 3C.$$

\paragraph{2ème cas : si $r \notin \gamma_1$.} Alors le segment $[p_1,r]\subset \gamma'_1$ contient $\gamma_1$ (car $\gamma_1$ et $\gamma_2$ sont orientés dans le même sens, donc $r$ ne peut pas être de l'autre côté de $p_1$ sur $\gamma'_1$). On a alors : $$|\gamma_1|+|\beta| \leq d(p_1,p_2) + |\gamma_2| + d(q_2,r) \leq |\gamma_1|+2C.$$
D'où : $$|\beta| \leq 2C.$$
Finalement : $$d(q_1,q_2) \leq d(q_1,r)+d(r,q_2) \leq |\beta| + C \leq 3C.$$
\end{proof}

\subsubsection*{Non-retournement}

\begin{prop} [Non-retournement]\label{copies}
Soit $B$ la $(2,4)$-constante de Morse du graphe des rayons et soit $w$ un sous-segment de l'axe $l=(\alpha_k)_{k\in \Z}$ de longueur supérieure à $10B$. Pour tout $g \in MCG(\R^2-K)$, si $g.w$ est inclus dans le $B$-voisinage de $l$, alors il est orienté dans le même sens que $w$.
\end{prop}

\noindent  Autrement dit, les segments de l'axe $l$ de longueur supérieure à $10B$ sont \emph{non retournables} : il n'existe pas de copies de $w^{-1}$ allant dans le sens de $w$ dans le $B$-voisinage de l'axe $l$.

\paragraph{Remarque :} Si un élément $h'\in \Gamma$ est conjugué à $h^{-1}$ par une application~{$\psi$}, alors on note $l'$ l'image de $l$ par $\psi$, que l'on munit de l'orientation inverse de celle de $l$. C'est un axe pour $h'$. D'après la proposition précédente, pour tout $w$ sous-segment de l'axe $l'$ de $h'$ de longueur supérieure à $10 B$ et orienté dans le même sens que $l'$, pour tout $g \in \Gamma$, si $g.w$ est inclus dans un $B$-voisinage de l'axe $l$ de $h$, alors $g.w$ est orienté dans le sens opposé à celui de $l$.

\paragraph{Preuve de la proposition \ref{copies}.}
On montre les deux lemmes suivants, qui nous permettent ensuite de conclure :

\begin{lemme} \label{w1}
Soient $m<n$ deux entiers positifs et soit $w=(\alpha_m,\alpha_{m+1},...,\alpha_n)$ un sous-segment de $(\alpha_k)_{k\in \N}$. Soit $g$ in élément de $MCG(\R^2-Cantor)$ tel que $d(\alpha_m,g.\alpha_n)\leq B$ et tel que $g.w$ est dans un $B$-voisinage de $l$, orienté dans le sens inverse de $w$. Alors si $|w|>8B+1$, il existe $m \leq i \leq n$ tel que $A(g.\alpha_{i+2})=A(g.\alpha_i)-2$.
\end{lemme}

\begin{proof}
Comme $d(\alpha_m,g.\alpha_n)\leq B$, on a $A(g.\alpha_n) \leq m+B$ (d'après le corollaire \ref{dist}).\\
Comme $g.w^{-1}$ et $w$ vont dans le même sens et ont même longueur, d'après le lemme \ref{géod-dist} on a :
$$d(\alpha_n,g.\alpha_m)\leq 3B.$$
D'où $A(g.\alpha_m) \geq n-3B$ (d'après le corollaire \ref{dist}).\\

Comme $A$ est $1$-lipschitzienne (lemme \ref{lemme distance}), $A(g.w)$ prend alors toutes les valeurs entières entre $m+B$ et $n-3B$. Par l'absurde, si on suppose que pour tout $i$ entre $m$ et $n$, $A(g.\alpha_{i+2}) \neq A(g.\alpha_i)-2$, alors pour tout $i$ on a : $$A(g.\alpha_{i+2}) \geq A(g.\alpha_i)-1.$$
Par récurrence, on en déduit :
$$A(g.\alpha_n)\geq A(g.\alpha_m)-{n-m\over 2}.$$
Comme $A(g.\alpha_m) \geq n-3B$ et $A(g.\alpha_m) \geq n-3B$, on a :
$$m+B\geq n-3B-{n-m\over 2}.$$
D'où finalement :
$$8B \geq n-m.$$
Comme on a supposé $|w| > 8B+1$, on obtient une contradiction.\end{proof}

\begin{lemme} \label{w2}
Pour tout $g \in MCG(\R^2-Cantor)$ et pour tout $i \geq 0$, on a :
$$A(g.\alpha_{i+2}) \neq A(g.\alpha_i)-2.$$
\end{lemme}

\begin{proof}
Comme pour tout $f \in MCG(\R^2 -Cantor)$ et pour tout $\beta, \gamma \in X_r$, on a $I(f.\beta,f.\gamma)=I(\beta,\gamma)$, on en déduit : $$I(g.\alpha_{i+2},g.\alpha_{i})=I(\alpha_2,\alpha_0)=0.$$

\noindent Par l'absurde, si $A(g.\alpha_{i+2}) = A(g.\alpha_i)-2$, d'après le lemme \ref{lemme-I1} on a : 
$$I(g.\alpha_{i+2},g.\alpha_{i}) \geq 1.$$
\end{proof}

On en déduit la proposition \ref{copies}:

\begin{proof}
Par l'absurde : supposons qu'il existe une copie de $w^{-1}$ qui convient, c'est-à-dire un $g \in \Gamma$ tel que le segment $g.w^{-1}=(g.\alpha_n,...,g.\alpha_{m+1},g.\alpha_m)$ est dans le $B$-voisinage de l'axe $l$ et va dans le même sens que $w$. Quitte à composer $g$ par $h^k$ pour un certain $k \in Z$, on peut supposer que $d(\alpha_m,g.\alpha_n) \leq B$. Comme $|w| > 8B+1$, les lemmes \ref{w1} et \ref{w2} nous permettent de conclure.
\end{proof}

\subsection{Un quasi-morphisme non trivial explicite sur $\Gamma$}

On rappelle la construction de Fujiwara de quasi-morphismes sur les groupes agissant sur des espaces hyperboliques (\cite{Fujiwara}). On fixe $p\in X_r$. Soient $w$ et $\gamma$ deux chemins de $X_r$. Une \emph{copie de $w$} est chemin de la forme $g.w$, avec $g\in \Gamma$. On note $|\gamma|_w$ le nombre maximal de copies disjointes de $w$ sur $\gamma$, et : $$c_w(g):=d(p,g(p))-\mathrm{inf}_\gamma (\mathrm{long}(\gamma)-|\gamma|_w),$$ l'infimum étant considéré sur tous les chemins $\gamma$ entre $p$ et $g(p)$. Comme $X_r$ est hyperbolique, on a alors que l'application $q_w : \Gamma \rightarrow \R$ définie par $q_w:=c_w -c_{w^{-1}}$ est un quasi-morphisme sur $\Gamma$ (proposition $3.10$ de \cite{Fujiwara}). De plus, le quasi-morphisme homogène $\tilde q_w$ défini par $\tilde q_w(g) = \lim_{n\rightarrow \infty} {q(g^n)\over n} $ ne dépend pas du point $p \in X_r$ choisi pour construire $c_w$.\\

On peut maintenant montrer la proposition suivante (qui n'est pas utile pour montrer que l'espace des quasi-morphismes non triviaux est de dimension infinie) :

\begin{prop}\label{prop-qm}
Soit $(\alpha_k)_{k\in \Z}$ la géodésique du graphe des rayons définie précédemment et soit $w$ un sous-segment de cette géodésique de longueur supérieure à $10B$, où $B$ est la $(2,4)$-constante de Morse du graphe des rayons. Alors le quasi-morphisme $\tilde q_w$ est non trivial.
\end{prop}

\paragraph{Remarque :} Comme on connait la constante d'hyperbolicité du graphe $X_\infty$, on peut en déduire celle du graphe des rayons, et on peut donc calculer la constante $B$ : ainsi le segment $w$ peut être choisi explicitement.

\begin{proof}
Comme $\tilde q_w$ est homogène, il suffit de montrer que ce n'est pas un morphisme pour avoir la non-trivialité. On montre d'une part que $\tilde q_w(h)$ est non nul, où $h=t_1t_2t_1$ est l'élément hyperbolique de $\Gamma$ défini précédemment, et d'autre part que $\tilde q_w(t_1)=\tilde q_w(t_2)=0$ : ainsi $\tilde q_w(t_1t_2t_1) \neq \tilde q_w(t_1)+\tilde q_w(t_2)+\tilde q_w(t_1)$, donc $\tilde q_w$ n'est pas un morphisme.\\

La première affirmation se déduit de la proposition \ref{copies}. C'est la stratégie décrite dans \cite{SCL}, page $74$ : si l'on note $m$ la longueur de $w$ et si l'on choisit $p =\alpha_0$, pour tout $k\in \N$ on a $c_w(h^{km})=k$ et $c_{w^{-1}}(h^{km})=0$. En effet, la première égalité est évidente, et pour la deuxième, on utilise le fait que les chemins qui réalisent l'infimum sont des $(2,4)$-géodésiques, d'après le lemme $3.3$ de \cite{Fujiwara}. Ils restent donc dans un $B$-voisinage de l'axe $(\alpha_k)_k$, d'après le lemme de Morse (\ref{Morse}). Par ailleurs ce voisinage ne contient pas de copie de $w^{-1}$, d'après la proposition \ref{copies} (voir \cite{SCL} partie $3.5$ pour plus de détails).
Ainsi : $$\tilde q_w(h^m) := \lim_{k\rightarrow \infty} {c_w(h^{km}) - c_{w^{-1}}(h^{km})\over k} = 1.$$
Donc $\tilde q_w$ est non nul.

Montrons que $\tilde q_w(t_1)=\tilde q_w(t_2)=0$. On choisit $p=\alpha_0$. Alors pour tout $k \in \N$, $t_1^k(\alpha_0)$ est une classe d'isotopie de courbe incluse dans l'hémisphère nord, donc $d(p,t_1^k(p))=1$. Ainsi $c_w(t_1^k)=c_{w^{-1}}(t_1^k)=0$, et donc $\tilde q_w(t_1)=0$. De la même façon, $\tilde q_w(t_2)=0$. Finalement, on a montré que $\tilde q_w$ est un quasi-morphisme non trivial. \end{proof}

\paragraph{Remarque :} Pour montrer que $\tilde q_w$ n'est pas un morphisme, on peut aussi montrer que \emph{$\Gamma$ est un groupe parfait}, c'est-à-dire que tout élément de $\Gamma$ s'écrit comme un produit de commutateurs. On en déduit que le seul morphisme de $\Gamma$ dans $\R$ est le morphisme trivial. Comme $\tilde q_w$ est non identiquement nul, ce n'est pas un morphisme. 

L'écriture de tout élément de $\Gamma$ comme produit de commutateurs se déduit du lemme de Calegari dans \cite{blog-Calegari}, que l'on peut énoncer ainsi : \emph{Si $g\in \Gamma$ est tel qu'il existe $x \in X_r$ tel que $d(x,gx)=1$, alors $g$ est le produit d'au plus deux commutateurs.} 

Soit $g\in \Gamma$ et soit $x\in X_r$ quelconque. On considère un chemin dans $X_r$ entre $x$ et $gx$, que l'on note  $(x=x_0,x_1,...,x_n=gx)$. Comme $\Gamma$ agit transitivement sur $X_r$, pour tout $1\leq i \leq n-1$ il existe $g_i \in \Gamma$ qui envoie $x_{i+1}$ sur $x_i$, et qui s'écrit donc comme produit d'au plus deux commutateurs. On a alors que $g_1...g_{n-1}g$ envoie $x$ sur $x_1$, avec $d(x,x_1)=1$. Ainsi cet élément s'écrit aussi comme produit d'au plus deux commutateurs. Finalement $g$ s'écrit comme produit de commutateurs.

\subsection{Dimension de l'espace des quasi-morphismes non triviaux}

\begin{theo}\label{dim infinie}
L'espace $\tilde{Q}(\Gamma)$ des quasi-morphisme non triviaux sur $\Gamma$ est de dimension infinie.
\end{theo}

\begin{proof}
On utilise le théorème $1$ de \cite{Bestvina-Fujiwara}. Comme $\Gamma$ agit par isométries sur le graphe des rayons qui est hyperbolique, si on trouve deux éléments hyperboliques $h_1, h_2 \in \Gamma$ agissant par translation sur des axes $l_1$ et $l_2$ respectivement, tels que $l_1$ et $l_2$ sont orientés dans le sens de cette action, et qui vérifient les deux propriétés suivantes, alors le théorème est démontré (voir \cite{Bestvina-Fujiwara}). Les deux propriétés à vérifier sont :

\begin{enumerate}
\item \og $h_1$ et $h_2$ sont indépendants \fg \ : la distance entre un demi-axe quelconque de $l_1$ et un demi-axe quelconque de $l_2$ est non bornée.
\item \og $h_1  \nsim h_2$ \fg : il existe une constante $C$ telle que pour tout segment $w$ de $l_2$ de longueur supérieure à $C$, pour tout $g \in \Gamma$, $g.w$ sort du $B$-voisinage de $l_1$ ou bien est orienté dans le sens inverse de $l_1$.
\end{enumerate} 

\noindent Trouvons donc deux éléments hyperboliques qui vérifient ces propriétés. On note $h_1 \in \Gamma$ l'élément $h$ qui agit par translation sur l'axe $(\alpha_k)_k$ défini précédemment. Soit $\phi \in \Gamma$ la classe de la rotation d'angle $\pi$ autour de l'infini. On suppose que $K$ est symétriquement disposé autour de $\infty$, de sorte que $\phi$ préserve $K$ et envoie chaque sous-ensemble de Cantor $K_i$ sur $K_{-i-1}$. Soit enfin $h_2:= \phi h_1^{-1} \phi^{-1}$. Alors $h_1 \nsim h_2$ d'après la proposition \ref{copies} et la remarque qui la suit (la constante $C:=10B$ convient, où $B$ est la constante de Morse). D'autre part, on va montrer que $h_1$ et $h_2$ sont indépendants, ce qui conclura la preuve.

On a montré dans le corollaire \ref{dist} que pour tout $n \geq 2$, tout rayon à distance inférieure ou égale à $(n-2)$ de $h_1^n(\alpha_0)$ commence par $\mathring \alpha_2$. De même tout rayon à distance inférieure ou égale à $(n-1)$ de $h_1^{n-1}(\alpha_0)$ commence par $\mathring \alpha_1$. 

On a un phénomène similaire pour $h_2$, $h_2^{-1}$ et $h_1^{-1}$. On note $\sigma$ la classe d'isotopie de la symétrie axiale par rapport à l'équateur. En particulier, $\sigma$ est égale à son inverse, fixe l'ensemble de Cantor $K$ et n'est pas un élément de $\Gamma$ car ne préserve pas l'orientation. De plus, comme $\phi$ est aussi égale à son inverse, on a :
$$h_2=\phi h_1^{-1} \phi^{-1} = \sigma h_1 \sigma^{-1}.$$
$$ h_2^{-1}= \phi h_1 \phi^{-1}.$$
$$ h_1^{-1}=\sigma \phi h_1 (\sigma \phi)^{-1}.$$

D'autre part, on a $\phi \alpha_{-1}=\sigma \alpha_0$ (voir figure \ref{h et compagnie}). Comme $\alpha_n=h_1^n(\alpha_0)$, on en déduit, d'après la troisième égalité qui précède, que pour tout $k\in \Z$ :
$$\phi \alpha_{-k-1}=\sigma \alpha_k.$$

Si l'on étend l'écriture en suites complètes aux rayons qui commencent dans l'hémisphère nord, en ajoutant par exemple nord ou sud dans l'écriture en segments du rayon, on peut coder les $\phi \alpha_k$. On en déduit alors, en utilisant le corollaire \ref{dist} et les égalités qui précèdent, que (voir figure \ref{h et compagnie}) :
\begin{itemize}
\item Pour tout $n \geq 2$, tout rayon à distance inférieure ou égale à $n-2$ de $h_2^n(\phi \alpha_0)=h_2^n(\sigma \alpha_{-1})=\sigma h_1^{n-1}(\alpha_0)$ commence par ${\phi \mathring \alpha_{-2}}=\mathring{\sigma \alpha_1}$.

\item Pour tout $n\geq 2$, tout rayon à distance inférieure ou égale à $n-2$ de $h_2^{-n}(\phi \alpha_0)=\phi(h_1^n\alpha_0)$ commence par $ {\phi \mathring \alpha_{2}}$. 

\item Pour tout $n\geq 2$, tout rayon à distance inférieure ou égale à $n-2$ de $h_1^{-n}(\alpha_0)=\sigma \phi h_1^n \phi \sigma \alpha_0=\sigma \phi h_1^n \alpha_{-1}$ commence par $\mathring \alpha_{-2} = \sigma \phi \mathring \alpha_1$.
\end{itemize}

\vspace{0.3cm}
\begin{figure}[!h]
\centering
\def\svgwidth{0.87\textwidth}
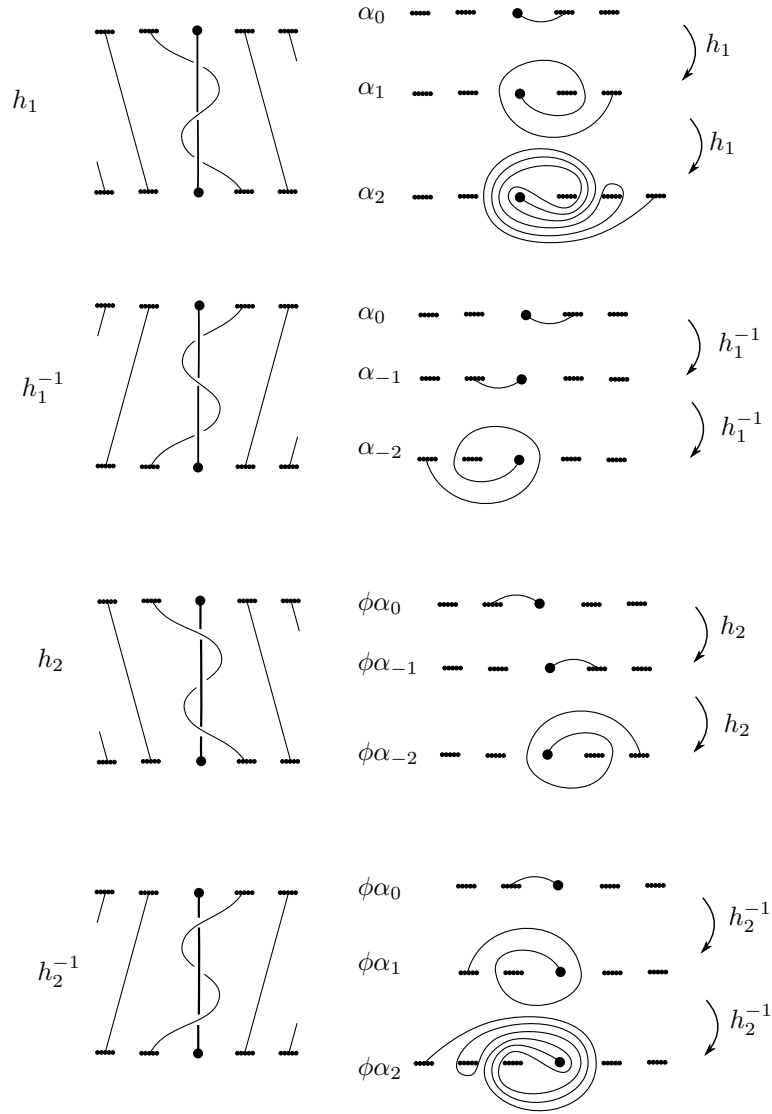
\vspace{0.3cm}
\caption{$h_1$, $h_2$, leurs inverses et leur action sur quelques rayons}
\label{h et compagnie}
\end{figure}

Ainsi, pour tout $n\geq 2$, tous les éléments des boules de rayon $(n-2)$ et de centres respectifs $h_1^n(\alpha_0)$, $h_2^n(\phi \alpha_0)$, $h_2^{-n}(\phi \alpha_0)$ et  $h_1^{-n}(\alpha_0)$ commencent respectivement par $\mathring \alpha_2$, ${\phi \mathring \alpha_{-2}}$, ${\phi \mathring \alpha_{2}}$ et $\mathring \alpha_{-2}$. Or $\mathring \alpha_2$, ${\phi \mathring \alpha_{-2}}$, ${\phi \mathring \alpha_{2}}$ et $\mathring \alpha_{-2}$ n'ont deux à deux pas de représentants disjoints : ces quatre boules sont donc disjointes (et même à distance supérieure à $1$). Ainsi les axes $l_1$ et $l_2$ de $h_1$ et $h_2$ sont tels que la distance entre deux demi-axes est non bornée. \end{proof}

\begin{figure}[!h]
\centering
\def\svgwidth{0.9\textwidth}
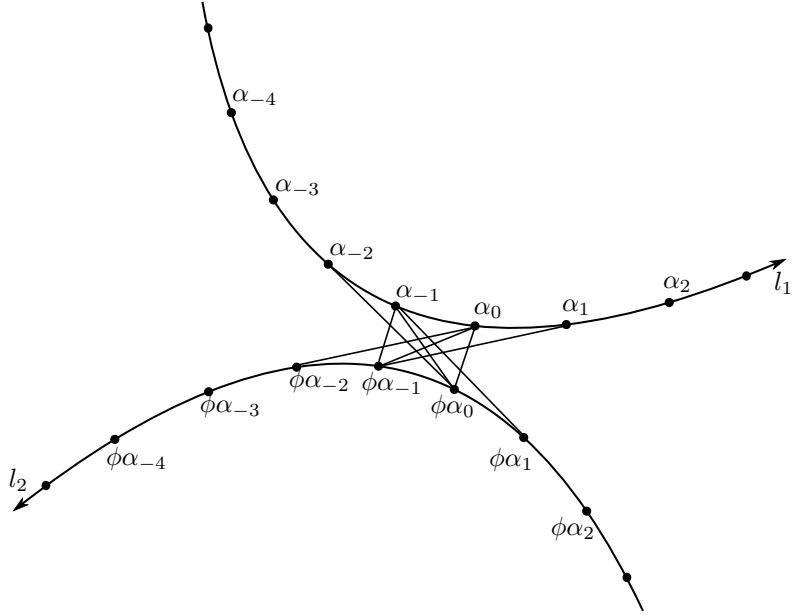
\caption{Axes de $h_1$ et $h_2$ : ce graphe est isométriquement plongé dans $X_r$}
\label{figu:axes}
\end{figure}

Plus précisément, on a que pour tous $|n|,|m| \geq 2$ (voir figure \ref{figu:axes}) : $$d(h_2^n(\phi \alpha_0), h_1^m(\alpha_0))\geq |n|+|m|-1.$$

\section{Exemple d'un élément hyperbolique de scl nulle}

Danny Calegari a montré que les éléments de $\Gamma$ ayant une orbite bornée sur le graphe des rayons sont de $scl$ nulle (voir \cite{blog-Calegari}). Montrons que la réciproque n'est pas vraie.

 \begin{prop}\label{ex-scl_nulle-hyp}
 Il existe un élément $g\in \Gamma$ hyperbolique de $scl$ nulle.
 \end{prop}

 \begin{proof} 
Soient $h_1$ et $h_2$ les deux éléments de $\Gamma$ définis dans la preuve du théorème \ref{dim infinie} : $h_1$ est l'élément $h$ défini plus tôt, et $h_2=\phi h_1^{-1} \phi^{-1}$, où $\phi$ est la classe de la rotation d'angle $\pi$ autour l'infini. Soit $g:=h_2h_1$ (voir figure \ref{homeo-g}). Alors $g$ est conjugué à son inverse (car $\phi=\phi^{-1}$), donc $scl(g)=0$. Montrons que de plus, $g$ est hyperbolique. On construit pour cela un demi-axe géodésique $(\gamma_k)_k$ du graphe des rayons, sur lequel $g$ agit par translation (comme on l'avait fait avec $(\alpha_k)_k$ pour montrer que $h$ est hyperbolique).

\begin{figure}[!h]
\labellist
\small\hair 2pt
\pinlabel $h_1$ at 235 130
\pinlabel $h_2$ at 235 40
\endlabellist
\centering
\vspace{0.3cm}
\includegraphics[scale=0.9]{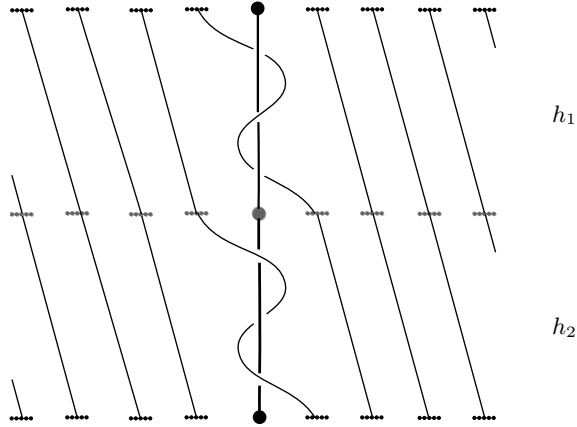}
\caption{L'élément $g:=h_2h_1$.}
\label{homeo-g}
\end{figure}

\paragraph{Définition de $(\gamma_k)_k$.}
La suite $(\gamma_k)_k$ se défini de manière similaire à $(\alpha_k)_k$, à ceci près que pour définir $\alpha_{k+1}$ à partir de $\alpha_k$ on longeait la courbe et on contournait le point d'attachement de $\alpha_k$ toujours par la droite, mais pour $(\gamma_k)_k$ on contourne le point d'attachement de la courbe précédente alternativement une fois par la droite, une fois par la gauche (voir figures \ref{gamma012} et \ref{gammak}).

\begin{figure}[!h]
\labellist
\small\hair 2pt
\pinlabel $\gamma_{0}$ at 59 70
\pinlabel $\gamma_{1}$ at 303 45
\pinlabel $\gamma_{2}$ at 188 -10
\endlabellist
\centering
\vspace{0.3cm}
\includegraphics[scale=0.8]{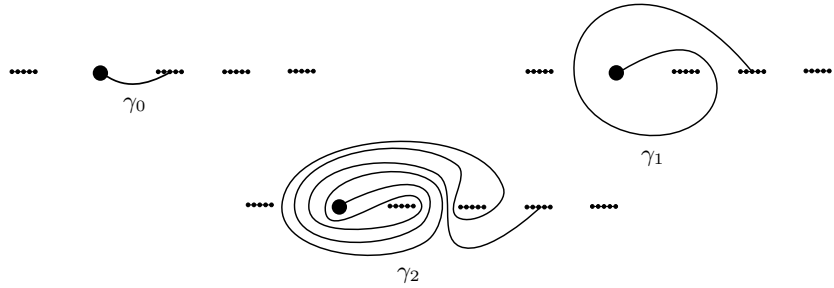}
\vspace{0.5cm}
\caption{Rayons $\gamma_0$, $\gamma_1$ et $\gamma_2$.}
\label{gamma012}
\end{figure}

Plus précisément, on définit la suite de rayons $(\gamma_k)_{k\geq0}$ par récurrence de la manière suivante :
\begin{itemize}
\item $\gamma_0:=\alpha_0$ est la classe d'isotopie du segment $s_0$ ayant pour extrémités $\infty$ et ${p}_0$.
\item Pour tout $k\geq 1$, $k$ impair (\emph{contournement du point d'attachement par la droite}) : pour obtenir $\gamma_{k+1}$, on part de $\infty$, on longe $\gamma_k$ jusqu'à son point d'attachement ${p}_k$ (dans un voisinage tubulaire de $\gamma_k$), on contourne ce point \emph{par la droite} en traversant les segments voisins, \emph{d'abord $s_{k+1}$ puis $s_{k}$}, on longe à nouveau $\gamma_k$ dans un voisinage tubulaire, on contourne $\infty$ en traversant $s_{0}$ puis $s_{-1}$, on longe une dernière fois $\gamma_k$ dans un voisinage tubulaire jusqu'à son point d'attachement et on va s'attacher au point ${p}_{k+1}$ sans traverser l'équateur.
\item Pour tout $k\geq 1$, $k$ pair (\emph{contournement du point d'attachement par la gauche}) : pour obtenir $\gamma_{k+1}$, on part de $\infty$, on longe $\gamma_k$ jusqu'à son point d'attachement ${p}_k$ (dans un voisinage tubulaire de $\gamma_k$), on contourne ce point \emph{par la gauche} en traversant les segments voisins, \emph{d'abord $s_{k}$ puis $s_{k+1}$}, on longe à nouveau $\gamma_k$ dans un voisinage tubulaire, on contourne $\infty$ en traversant $s_{-1}$ puis $s_{0}$, on longe une dernière fois $\gamma_k$ dans un voisinage tubulaire jusqu'à son point d'attachement et on va s'attacher au point ${p}_{k+1}$ sans traverser l'équateur.
\end{itemize}

\begin{figure}[!h]
\labellist
\small\hair 2pt
\pinlabel $\gamma_{2n-1}$ at -15 100
\pinlabel $\gamma_{2n}$ at -10 20
\pinlabel $K_{2n-1}$ at 336 89
\pinlabel $K_{2n}$ at 371 25
\pinlabel $1$ at 299 83
\pinlabel $3$ at 256 77
\pinlabel $3^{2n-5}$ at 174 77
\pinlabel $3^{2n-4}$ at 132 85
\pinlabel $3^{2n-3}$ at 87 79

\pinlabel $3^{2n-2}$ at 30 75
\pinlabel $1$ at 301 123
\pinlabel $3$ at 253 113
\pinlabel $3^{2n-5}$ at 178 115
\pinlabel $3^{2n-4}$ at 134 124
\pinlabel $3^{2n-3}$ at 91 113
\pinlabel $3^{2n-1}$ at 30 -7
\pinlabel $1$ at 336 -4
\endlabellist
\centering
\vspace{0.4cm}
\includegraphics[scale=0.8]{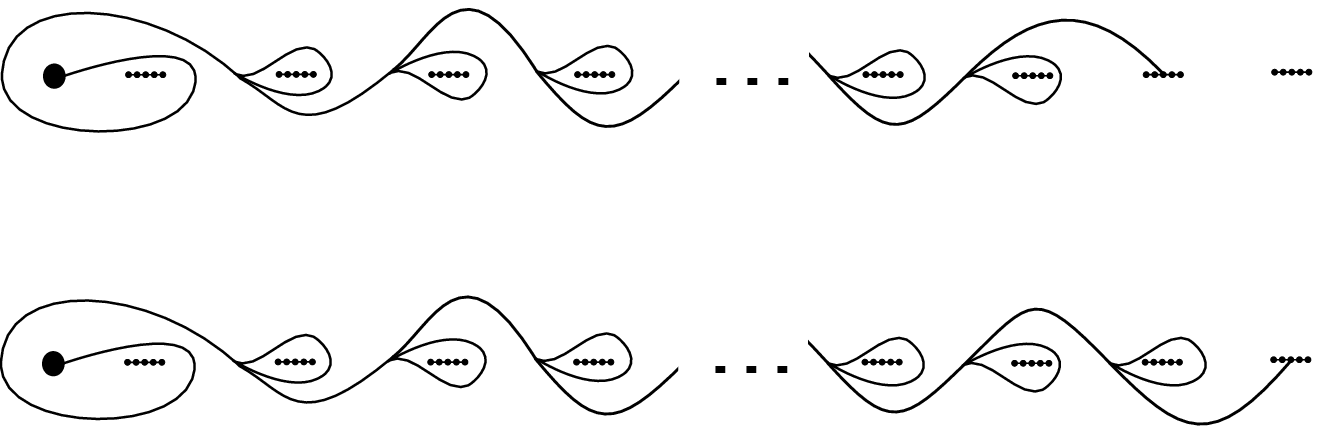}
\vspace{0.4cm}
\caption{Graphes représentant $\gamma_k$ dans le cas impair (en haut) et dans le cas $k$ pair (en bas).}
\label{gammak}
\end{figure}

Vues les similarités de construction entre $(\alpha_k)_k$ et $(\gamma_k)_k$, on peut adapter les mêmes arguments que ceux utilisées dans la section \ref{section1} : en particulier le lemme \ref{lemme distance}, puis son corollaire \ref{dist} et la proposition \ref{alpha_k geod}. On en déduit que le demi-axe $(\gamma_k)_{k\in \N}$ est géodésique.

\paragraph{L'élément $g$ agit par translation sur $(\gamma_k)_k$.} On voit par récurrence en utilisant les graphes représentant les rayons (comme on l'avait fait pour montrer que $h$ est hyperbolique) que $g^n(\gamma_0)=\gamma_{2n}$ pour tout $n\in \N$, voir la figure \ref{diag-g-gamma0} pour le cas $n=0$, la figure \ref{diag-g-gamma2} pour le cas $n=1$ et la figure \ref{diag-g-gammak} pour le cas général.
\end{proof}


\begin{figure}[!h]
\labellist
\small\hair 2pt
\pinlabel $\gamma_0$ at 101 349
\pinlabel $\gamma_2$ at 97 -9
\pinlabel $h_1$ at 235 328
\pinlabel $h_2$ at 278 154
\pinlabel $g$ at 391 249
\pinlabel $3$ at 397 -8
\pinlabel $1$ at 473 -5
\pinlabel $1$ at 473 39
\endlabellist
\centering
\includegraphics[scale=0.65]{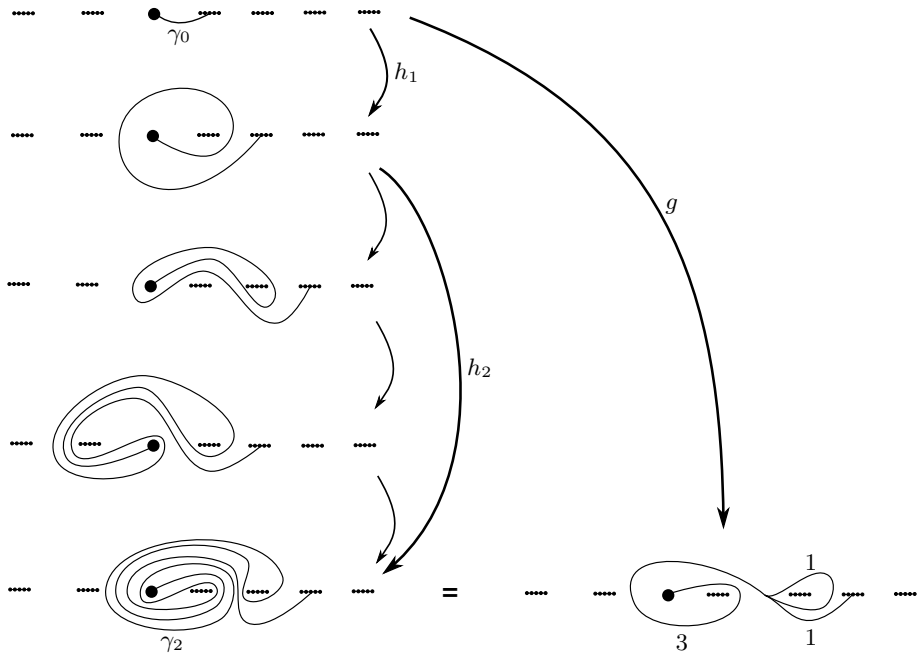}
\vspace{0.1cm}
\caption{Graphe représentant l'image de $\gamma_0$ par $g$, où l'on a décomposé l'action de $h_2$ en trois parties, comme on l'avait fait plus tôt pour $h$. Les deux graphes du bas représentent le même rayon, à savoir $g(\gamma_0)=\gamma_2$.}
\label{diag-g-gamma0}
\end{figure}


\begin{figure}[!h]
\labellist
\small\hair 2pt
\pinlabel $h_1$ at 385 122
\pinlabel $h_2$ at 385 45
\pinlabel $3$ at 91 128
\pinlabel $1$ at 178 141
\pinlabel $3^2$ at 89 61
\pinlabel $3$ at 180 92
\pinlabel $1$ at 226 76
\pinlabel $3^3$ at 99 -9
\pinlabel $3^2$ at 184 17
\pinlabel $3$ at 230 15
\pinlabel $1$ at 278 9

\pinlabel $\gamma_2$ at -10 152
\pinlabel $\gamma_4$ at -10 16
\endlabellist
\centering
\vspace{0,7cm}
\includegraphics[scale=0.85]{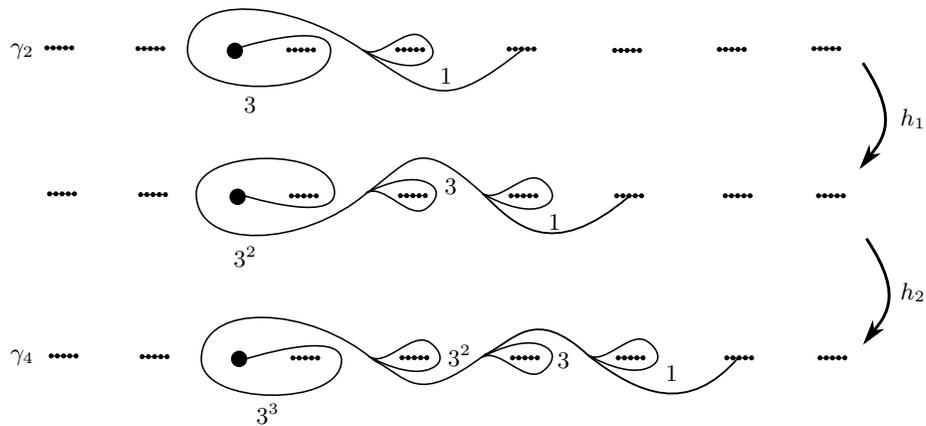}
\vspace{0,5cm}
\caption{Action de $g$ sur $\gamma_2$.}
\label{diag-g-gamma2}
\end{figure}

\begin{figure}[!h]
\labellist
\small\hair 2pt
\pinlabel $\gamma_{2n}$ at 6 183
\pinlabel $\gamma_{2n+2}$ at 9 13
\pinlabel $3^{2n-1}$ at 58 169
\pinlabel $3^{2n}$ at 58 84
\pinlabel $3^{2n+1}$ at 58 -7
\pinlabel $1$ at 327 165
\pinlabel $1$ at 367 85
\pinlabel $1$ at 415 -9
\pinlabel $K_{2n}$ at 357 203
\pinlabel $K_{2n+1}$ at 398 118
\pinlabel $K_{2n+2}$ at 442 31
\pinlabel $h_1$ at 518 153
\pinlabel $h_2$ at 520 64
\pinlabel $3^{2n-2}$ at 127 174
\pinlabel $3^{2n-1}$ at 125 93
\pinlabel $3^{2n}$ at 127 -2
\pinlabel $3^{2n-2}$ at 178 85
\pinlabel $3^{2n-2}$ at 225 -5
\pinlabel $3^{2n-1}$ at 175 3

\endlabellist
\centering
\includegraphics[scale=0.7]{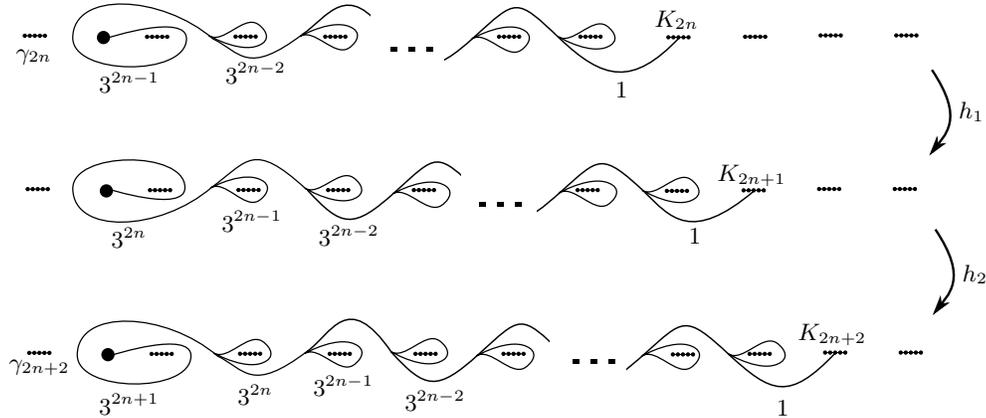}
\vspace{0.1cm}
\caption{Action de $g$ sur $\gamma_{2n}$ : $g(\gamma_{2n}) = \gamma_{2n+2}$.}
\label{diag-g-gammak}
\end{figure}

\bibliographystyle{alpha}
\bibliography{ref}

\end{document}